\documentclass[11pt,a4,onesided]{amsart}

\usepackage{hyperref}
\usepackage{amsfonts}
\usepackage{amssymb}
\usepackage{amsmath}
\usepackage{amsthm}
\usepackage{latexsym}
\usepackage{color}
\usepackage{amscd}
\usepackage[all,cmtip]{xy}
\usepackage{dsfont}
\usepackage{mathrsfs}
\usepackage{graphicx}
\usepackage{enumitem} 
\usepackage[margin=1.6in]{geometry}
\usepackage{esint}
\usepackage{epsfig}

\allowdisplaybreaks

\numberwithin{equation}{section}
\newcounter{exercise}

\newtheorem{theorem}{Theorem}
\newtheorem{proposition}[theorem]{Proposition}
\newtheorem{lemma}[theorem]{Lemma}

\newtheorem{remark}[theorem]{Remark}

\theoremstyle{definition}
\newtheorem{definition}[theorem]{Definition}

\theoremstyle{remark}

\newcommand\R{{\mathbb R}}

\newcommand{\nx}{\nabla_x}
\newcommand{\ny}{\nabla_y}
\newcommand{\iy}{\int_Y}
\newcommand{\iti}{\int_0^T}
\newcommand{\iR}{\int_{\R^d}}

\def\ts{\stackrel{2-drift}{\relbar\joinrel\relbar\joinrel\relbar\joinrel\relbar\joinrel\rightharpoonup}}
\def\tss{\stackrel{2s-drift}{\relbar\joinrel\relbar\joinrel\relbar\joinrel\relbar\joinrel\rightharpoonup}}

\def\div{{\rm div}}
\def\dsp{\displaystyle}

\let\oldmarginpar\marginpar
\renewcommand\marginpar[1]{\-\oldmarginpar[\raggedleft\footnotesize #1]%
{\raggedright\footnotesize #1}}

\newcommand{\ve}{\varepsilon}

\renewcommand{\a}{\alpha}
\renewcommand{\b}{\beta}
\renewcommand{\v}{\varphi}

\def\signga{\bigskip {\small \sc
      CMAP, UMR CRNS 7641, \'Ecole Polytechnique,
      Route de Saclay, Palaiseau F91128, France \par E-mail:}
    \tt{\small gregoire.allaire@polytechnique.fr} }

\def\signhh{\bigskip {\small \sc
      DPMMS, CMS, University of Cambridge,
      Wilberforce road, Cambridge CB3 0WB, UK
      \par E-mail:}
    \tt{\small H.Hutridurga@dpmms.cam.ac.uk} }

\begin{document}

\title[Multicomponent Transport]{On the homogenization of multicomponent
transport}

\author{Gr\'egoire Allaire, Harsha Hutridurga}

\date{\today}

\maketitle

\begin{abstract}
This paper is devoted to the homogenization of weakly coupled cooperative
parabolic systems in strong convection regime with purely periodic
coefficients. Our approach is to factor out oscillations from the solution via
principal eigenfunctions of an associated spectral problem and to cancel any
exponential decay in time of the solution using the principal eigenvalue of the
same spectral problem. We employ the notion of two-scale convergence with drift
in the asymptotic analysis of the factorized model as the lengthscale of the
oscillations tends to zero. This combination of the factorization method and
the method of two-scale convergence is applied to upscale an adsorption model
for multicomponent flow in an heterogeneous porous medium.
\end{abstract}


\section{Introduction}
\label{sec:int}
Upscaling reactive transport models in porous media is a problem of great practical
importance and homogenization theory is a method of choice for achieving this goal
(see \cite{HOR} and references therein). In this paper we focus on a model problem
of reactive multicomponent transport for $N$ diluted chemical species in a saturated
periodically varying media. The fluid velocity is assumed to be known.
On top of usual convective and diffusive effects we consider linear reaction terms
which satisfy a specific condition, namely that the reaction matrix is cooperative
(see the precise definition in Section \ref{sec:model}). This assumption is quite 
natural for a linear system, as we consider here, since it ensures a maximum (or 
positivity) principle for solutions which, being concentrations, should indeed be 
non-negative for obvious physical reasons. As usual the ratio between the period 
of the coefficients and a characteristic lengthscale of the porous domain is 
denoted by a small parameter $0<\ve\ll1$. Denoting the unknown concentrations 
by $u^\ve_\a$, for $1\le\a\le N$, we study in the entire space $\R^d$ the following
weakly coupled (i.e., no coupling in the derivatives) system of $N$ parabolic
equations with periodic bounded coefficients:
\begin{equation}
\label{eq:intro}
\rho_\a\Big(\frac{x}{\ve}\Big)\frac{\partial u^\ve_\a}{\partial t} +
\frac{1}{\ve}b_\a\Big(\frac{x}{\ve}\Big)\cdot\nabla u^\ve_\a -
\div\Big(D_\a\Big(\frac{x}{\ve}\Big)\nabla u^\ve_\a\Big) +
\frac{1}{\ve^2}\sum_{\b=1}^N\Pi_{\a\b}\Big(\frac{x}{\ve}\Big)u^\ve_\b = 0,
\end{equation}
for $1\le\a\le N$, where $b_\a$ are velocity fields, $D_\a$ are symmetric and
coercive diffusion tensors and $\Pi$ is the reaction (or coupling) matrix,
assumed
to be cooperative (see (\ref{eq:ass:irr}) for a precise definition).
All coefficients are $Y$-periodic, where $Y:=]0,1[^d$ is the unit cell in
$\R^d$.
Our main result, Theorem \ref{thm:hom}, states that a solution to the Cauchy
problem
for (\ref{eq:intro}) admits the following asymptotic representation (for every
$1\le\a\le N$):
$$
u^\ve_\a(t,x)=\v_\a\Big(\frac{x}{\ve}\Big)\exp{(-\lambda t/\ve^2)}\Big(v\Big(t,
x-\frac{b^*t}{\ve}\Big)+\mathcal{O}(\ve)\Big),
$$
where $\{\lambda,(\varphi_\a)_{1\le\a\le N}\}$ is the first eigenpair for a
periodic system
posed in the unit cell $Y:=]0,1[^d$, $b^*$ is a constant drift vector and
$v(t,x)$ solves
a scalar parabolic homogenized problem with constant coefficients. Our result
generalizes the works \cite{AR:07} and \cite{DP:05}, which were restricted to a
single (scalar) parabolic equation. In \cite{Ca:00}, \cite{Ca:02} a similar
result was obtained for a cooperative elliptic system without convective terms.
Our present work is thus the first to combine large convective terms and
multiple equations.

Let us explain the specific $\ve$-scaling of the coefficients in
(\ref{eq:intro}), which
is not new and is well explained, e.g., in \cite{AMP}. Before
adimensionalization, the
physical system of equations, in original time-space coordinates $(\tau,y)$,
is,
for $1 \le \alpha \le N$,
$$
\rho_\a\frac{\partial u_\a}{\partial\tau} + b_\a\cdot\nabla u_\a -
\div(D_\a\nabla u_\a) + \sum_{\b=1}^N \Pi_{\a\b}u_\b = 0 .
$$
Interested by a macroscopic view and long time behaviour of this parabolic
system,
we perform a ``parabolic'' scaling of the time-space variables, i.e.,
$(\tau,y)\to(\ve^{-2}t,\ve^{-1}x)$, which yields the scaled model
(\ref{eq:intro}).

\begin{remark}
Another scaling that one could consider is the ``hyperbolic'' scaling, 
i.e., $(\tau,y)\to(\ve^{-1}t,\ve^{-1}x)$. This has been
addressed in \cite{MS:13} (for $N=2$) where the scaled system is:
$$
\rho_\a\Big(\frac{x}{\ve}\Big)\frac{\partial u^\ve_\a}{\partial t} + 
b_\a\Big(\frac{x}{\ve}\Big) \cdot\nabla u^\ve_\a - 
\ve\div\Big(D_\a\Big(\frac{x}{\ve}\Big)\nabla u^\ve_\a\Big) + 
\frac1{\ve}\sum_{\b=1}^N \Pi_{\a\b}\Big(\frac{x}{\ve}\Big)u^\ve_\b = 0 ,
$$
for $1\le\a\le N$. The main result of \cite{MS:13} is that the solution to the
Cauchy problem for the above system admits the asymptotic representation:
$$
u^\ve_\a(t,x) \approx \phi_\a\Big(\frac{x}{\ve}\Big)\delta(x-b^*t)
$$
where $\phi_\a$ is the first eigenfunction and there is no time exponential 
because $\lambda=0$ happens to be the first eigenvalue for the specific 
choice of cooperative matrix $\Pi_{\a\b}$ made in \cite{MS:13}. In the 
above equation $\delta$ is the Dirac mass which appears because of a 
concentration assumption on the initial data. The main difference 
with the parabolic scaling in our work is that there is no diffusion 
homogenized problem. The drift velocity can be interpreted as $b^*=\nabla H(0)$ 
with $H$ being some effective Hamiltonian. 
\end{remark}

The organization of this paper is as follows. In Section \ref{sec:model}, we
describe the mathematical model of cooperative parabolic systems and the
precise hypotheses made on the coefficients. Section \ref{sec:qa} briefly 
recalls the existence and uniqueness theory for system (\ref{eq:intro}). 
Since no uniform a priori estimates can be obtained for (\ref{eq:intro}), 
a factorization principle (or change of unknowns) is performed in 
Section \ref{sec:fp}. Then, uniform a priori bounds are deduced for 
the solution of this factorized problem. The definition of two-scale 
convergence with drift is recalled in Section \ref{sec:2scl}. Then, 
based on the uniform a priori estimates of Section \ref{sec:fp}, 
we obtain a two-scale compactness result for the sequence of solutions 
(see Theorem \ref{thm:3:2scl}). Our main homogenization result is 
Theorem \ref{thm:hom} which is proved in Section \ref{sec:hom}. Eventually, 
Section \ref{sec:apm} generalizes our analysis to a similar, but more 
involved, system which is meaningful from a physical point of view. 
The differences are that $(i)$ the convection-diffusion takes place 
in a perforated porous medium and $(ii)$ the chemical reactions are 
localized on the holes' boundaries rather than in the fluid bulk. 
This is a frequent case for adsorption or deposition of the chemical 
on the solid surface (cf. the discussion and references in \cite{AMP}).

\section{The model}
\label{sec:model}
Before we present our model, let us introduce the following shorthands:
$$
\rho^\ve_\a(x):=\rho_\a\Big(\frac{x}{\ve}\Big);\hspace{0.5
cm}b^\ve_\a(x):=b_\a\Big(\frac{x}{\ve}\Big);\hspace{0.5
cm}D^\ve_\a(x):=D_\a\Big(\frac{x}{\ve}\Big);\hspace{0.5
cm}\Pi_{\a\b}^\ve(x):=\Pi_{\a\b}\Big(\frac{x}{\ve}\Big),
$$
where the small positive parameter $\ve\ll1$ represents the lengthscale of
oscillations. We consider the following Cauchy problem:
\begin{equation}
\label{eq:cd}
\dsp\rho^\ve_\a \frac{\partial u^\ve_\a}{\partial t} + \frac1\ve
b^\ve_\a\cdot\nabla u^\ve_\a - \div(D^\ve_\a\nabla u^\ve_\a) +
\frac1{\ve^2}\sum_{\b=1}^N \Pi^\ve_{\a\b}u^\ve_\b = 0\hspace{0.1 cm}\mbox{ in
}\hspace{0.1 cm}(0,T)\times\R^d,
\end{equation}
\begin{equation}
\label{eq:in}
u^\ve_\a(0,x)=u^{in}_\a(x)\hspace{0.1 cm}\mbox{ for }x\in\R^d.
\end{equation}
For a normed vector space $\mathcal H$,
we use the following standard notation for $Y$-periodic function spaces:
$$
L^p_\#(\R^d;\mathcal H):=\Big\{f:\R^d\to\mathcal H \mbox{ s.t. }f\mbox{ is
}Y\mbox{-periodic}
\mbox{ and } \|\|f\|_{\mathcal H}\|_{L^p(Y)}<\infty \Big\}.
$$
The assumptions made on the coefficients of (\ref{eq:cd}) are the following:
\begin{equation}
\label{eq:ass:rho}
\rho_\a\in L^\infty_\#(\R^d;\R)\mbox{ and }\exists c_\a>0\mbox{ s.t. }
\rho_\a(y)\geq c_\a ,
\end{equation}
\begin{equation}
\label{eq:ass:vel}
b_\a\in L^\infty_\#(\R^d;\R^d) \mbox{ and } \div b_\a\in L^\infty_\#(\R^d;\R),
\end{equation}
\begin{equation}
\label{eq:ass:diff}
D_\a=(D_\a)^*\in L^\infty_\#(\R^d;\R^{d\times d})\mbox{ and }\exists
c_\a>0\mbox{ s.t. }c_\a|\xi|^2\le D_\a(y)\xi\cdot\xi
\end{equation}
for all $\xi\in\R^d$ and for almost every $y\in\R^d$ (where $(D_\a)^*$ 
is the adjoint or transposed matrix of $D_\a$),
\begin{equation}
\label{eq:ass:cple}
\Pi\in L^\infty_\#(\R^d;\R^{d\times d})\mbox{ and }\Pi_{\a\b}\le0\mbox{ for
}\a\not=\b,
\end{equation}
we also assume that the coupling matrix $\Pi$ is irreducible, i.e., there
exists no partition
$\mathcal{B}\not=\emptyset$,$\mathcal{B}'\not=\emptyset$ of $\{1,\cdots,N\}$
such that
\begin{equation}
\label{eq:ass:irr}
\{1,\cdots,N\}=\mathcal{B}\cup\mathcal{B}'\mbox{ with }
\mathcal{B}\cap\mathcal{B}'=\emptyset \mbox{ and }\Pi_{\a\b}=0\mbox{ for all
}\a\in\mathcal{B}, \b\in\mathcal{B}'.
\end{equation}
This irreducibility assumption ensures that the system (\ref{eq:cd}) cannot be
decoupled in two disjoint subsystems (see Remark \ref{rem:irr} below).

\begin{remark}
The only assumption made on the convective fields $b_\a$ in (\ref{eq:ass:vel})
is that they are bounded as well as their divergences. No divergence-free assumption is made on
these vector fields. The hypotheses (\ref{eq:ass:cple})-(\ref{eq:ass:irr}) are
borrowed from \cite{Sw:92, MS:95, AC:00, Ca:02}. A matrix satisfying
(\ref{eq:ass:cple}) is sometimes referred to as ``cooperative matrix'' (up to
the addition of a multiple of the identity it is also an $M$-matrix). Hence the
system (\ref{eq:cd}) gets the name ``cooperative parabolic system''.
\end{remark}

Finally, we assume that the initial data in (\ref{eq:in}) has following
regularity: $u^{in}_\a\in L^2(\R^d)$ for each $1 \le \alpha \le N$.

\section{Qualitative Analysis}
\label{sec:qa}
Results of existence and uniqueness of solutions to (\ref{eq:cd}) are
classical.
The ``cooperative'' hypothesis (\ref{eq:ass:cple}) is actually not necessary to
obtain well-posedness. Standard approach is to derive a priori estimates on the
solution.
Classical technique is to multiply (\ref{eq:cd}) by $u^\ve_\a$ and integrate
over $\R^d$:
$$
\frac12\frac{d}{dt}\iR\rho_\a^\ve|u^\ve_\a|^2{\rm d}x+\iR D^\ve_\a\nabla
u^\ve_\a\cdot\nabla u^\ve_\a{\rm d}x
$$
$$
=\frac{1}{2\ve}\iR\div(b^\ve_\a)|u^\ve_\a|^2{\rm d}x -
\frac1{\ve^2}\sum_{\b=1}^N\iR \Pi^\ve_{\a\b}u^\ve_\b u^\ve_\a {\rm d}x.
$$
Since the divergences of the convective fields are bounded, summing the above
expression over $1\le\a\le N$ followed by the application of Cauchy-Schwarz
inequality, Young's inequality, Gronwall's lemma and an integration over
$(0,T)$ leads to the following a priori estimates:
\begin{equation}
\label{eq:apr:ve}
\dsp\sum_{\a=1}^N\|u^\ve_\a\|_{L^\infty((0,T);L^2(\R^d))} +
\sum_{\a=1}^N\|\nabla u^\ve_\a\|_{L^2((0,T)\times\R^d)}\le
C_\ve\sum_{\a=1}^N\|u^{in}_\a\|_{L^2(\R^d)},
\end{equation}
where the constant $C_\ve$ depends on the small parameter $\ve$. For any fixed
$0<\ve$, we can use the a priori estimates (\ref{eq:apr:ve}) and Galerkin
method to establish existence and uniqueness of the solution $u^\ve_\a\in
L^2((0,T);H^1(\R^d))\cap C((0,T);L^2(\R^d))$, $1\le\a\le N$.

Maximum principles are a different story altogether. In general we have no
maximum principles for systems. However, the hypotheses
(\ref{eq:ass:cple})-(\ref{eq:ass:irr}) guarantee a maximum principle. In
\cite{Sw:92, MS:95}, weakly coupled cooperative elliptic systems with coupling
matrices satisfying (\ref{eq:ass:cple})-(\ref{eq:ass:irr}) are studied with
emphasis on maximum principles and on the well-posedness of associated spectral
problems. The results of \cite{Sw:92} on the cooperative elliptic systems have
their parabolic counterpart. We state the result from \cite{Sw:92} adapted to
cooperative parabolic systems:

\begin{lemma}[see \cite{Sw:92, MS:95} for a proof]
Let the conditions (\ref{eq:ass:rho})-(\ref{eq:ass:irr}) on the coefficients of
(\ref{eq:cd}) be satisfied. Then, for any fixed $\ve>0$, the following holds:

(i) There is a unique solution $u^\ve_\a\in L^2((0,T);H^1(\R^d))\cap
C((0,T);L^2(\R^d))$ for $1 \le \alpha \le N$.

(ii) If $u^{in}_\a\ge0$ for all $1 \le \alpha \le N$, then $u^{\ve}_\a\ge0$ for
all $1 \le \alpha \le N$.
\end{lemma}

\begin{remark}
In order to make an asymptotic analysis on (\ref{eq:cd}), as $\ve\to0$, one
demands uniform (with respect to $\ve$) estimates on the solution $u^\ve_\a$.
But the estimates in (\ref{eq:apr:ve}) are not uniform in $\ve$. This renders
the application of standard compactness theorems from homogenization theory
useless for (\ref{eq:cd}).
\end{remark}

\section{Factorization Principle}
\label{sec:fp}
The difficulty with the derivation of a priori estimates in presence of large
lower order terms has long been recognized \cite{Va:81, Ji:84, Ko:84, AC:00, ACPSV:04,
DP:05}. The idea is to use information from an associated spectral cell
problem. The basic principle is to factor out principal eigenfunction from the
solution to arrive at a new ``factorized system'', amenable to uniform a priori
estimates. This idea of factoring our oscillations from the solution was first
introduced in \cite{Va:81} in the context of elliptic eigenvalue problems. In
case of scalar parabolic equations it is shown in \cite{Ji:84, Ko:84, DP:05,
AR:07} that the factorized equations have no zero order terms and that the
first order terms are divergence free. In case of cooperative elliptic systems
with large lower order terms studied in \cite{AC:00, Ca:02}, however, it is
shown that the factorized systems still have zero order terms and are
transformed as ``difference terms''. We adopt the ``factorization principle'',
extensively used in the above mentioned references, to remedy the difficulty we
have with the derivation of uniform a priori estimates for (\ref{eq:cd}). We
first define the following spectral problem associated with (\ref{eq:cd}) and
posed in the unit cell with periodic boundary conditions:
\begin{equation}
\label{eq:scp}
\left\{
\begin{array}{ll}
\dsp b_\a \cdot \ny \v_\a - \div_y \Big(D_\a \ny \v_\a\Big) + \sum_{\b=1}^N
\Pi_{\a\b} \v_\b = \lambda\rho_\a \v_\a & \mbox{ in }Y, \\[0.3cm]
y \to \v_\a(y)  \hspace{1 cm} Y\mbox{-periodic.}
\end{array} \right.
\end{equation}
The above spectral cell problem is not self-adjoint. The associated adjoint
problem is:
\begin{equation}
\label{eq:ascp}
\left\{
\begin{array}{ll}
\dsp -\div_y(b_\a \v^*_\a) - \div_y \Big(D_\a \ny \v^*_\a\Big) + \sum_{\b=1}^N
\Pi^*_{\a\b} \v^*_\b = \lambda\rho_\a \v^*_\a & \mbox{ in }Y,\\[0.3cm]
y \to \v^*_\a(y) \hspace{1 cm} Y\mbox{-periodic,}
\end{array} \right.
\end{equation}
where $\Pi^*$ is the transpose of $\Pi$. The well-posedness of the above
spectral problems is a delicate issue which is addressed in \cite{Sw:92, MS:95}. 
The following proposition is an adaptation to our periodic setting of the 
main result of \cite{Sw:92, MS:95}.

\begin{proposition}[see \cite{Sw:92} for a proof]
\label{prop:spec}
Under the assumptions (\ref{eq:ass:rho})-(\ref{eq:ass:irr}) on the
coefficients, the spectral problems (\ref{eq:scp}) and (\ref{eq:ascp}) admit a
common first eigenvalue (i.e., smallest in modulus) which satisfies:

(i) the first eigenvalue $\lambda$ is real and simple,

(ii) the corresponding first eigenfunctions $(\v_\a)_{1\le\a\le N}\in
(H^1_\#(Y))^N$ for (\ref{eq:scp}), $(\v^*_\a)_{1\le\a\le N}\in (H^1_\#(Y))^N$
for (\ref{eq:ascp}) are positive, $\v_\a,\v^*_\a>0$ for $1\le\a\le N$, and
unique up to normalization.
\end{proposition}

\begin{remark}
\label{rem:spec}
The first eigenvalue $\lambda$ in Proposition \ref{prop:spec} measures the
balance between convection-diffusion and reaction. Also, the uniqueness of
first eigenfunctions in Proposition \ref{prop:spec} is up to a chosen
normalization. The normalization that we consider is the following:
\begin{equation}
\label{eq:norm}
\sum_{\a=1}^N\:\:\int_Y \rho_\a \v_\a\v^*_\a{\rm d}y = 1.
\end{equation}
\end{remark}

In the proof of our a priori estimates it will be convenient to scale the 
spectral problems (\ref{eq:scp})-(\ref{eq:ascp}) to the entire domain $\R^d$ 
via the change of variables $y\to\ve^{-1}x$. More precisely, (\ref{eq:scp})-(\ref{eq:ascp}) 
are equivalent to

\begin{equation}
\label{eq:scp:scl}
\left\{
\begin{array}{ll}
\dsp \ve b^\ve_\a \cdot \nabla \v^\ve_\a - \ve^2\div \Big(D^\ve_\a \nabla \v^\ve_\a\Big) + \sum_{\b=1}^N
\Pi^\ve_{\a\b} \v^\ve_\b = \lambda\rho^\ve_\a \v^\ve_\a & \mbox{ in }\R^d, \\[0.3cm]
x \to \v^\ve_\a(x)\equiv \v_\a(x/\ve)  \hspace{1 cm} \ve Y\mbox{-periodic,}
\end{array} \right.
\end{equation}

\begin{equation}
\label{eq:ascp:scl}
\left\{
\begin{array}{ll}
\dsp -\ve\div(b^\ve_\a \v^{*\ve}_\a) - \ve^2\div \Big(D^\ve_\a \nabla \v^{*\ve}_\a\Big) + \sum_{\b=1}^N
\Pi^{*\ve}_{\a\b} \v^{*\ve}_\b = \lambda\rho^\ve_\a \v^{*\ve}_\a & \mbox{ in }\R^d,\\[0.3cm]
x \to \v^{*\ve}_\a(x)\equiv\v^*_\a(x/\ve) \hspace{1 cm} \ve Y\mbox{-periodic.}
\end{array} \right.
\end{equation}

Now, we get down to the task of reducing (\ref{eq:cd}) to a ``factorized
system''. As explained in \cite{AC:00, ACPSV:04, DP:05, AR:07} the first eigenvalue
$\lambda$ governs the time decay or growth of the solution $u^\ve_\a$. So, as
is done in the references cited, we perform time renormalization in the spirit
of the factorization principle. Also the first eigenfunction $\v^\ve_\a$ is
factored out of $u^\ve_\a$.  In other words we make the following change of
unknowns:
\begin{equation}
\label{eq:fcp}
\dsp v^\ve_\a(t,x)=\exp{(\lambda
t/\ve^2)}\frac{u^\ve_\a(t,x)}{\v^\ve_\a(x)}.
\end{equation}
The above change of unknowns is valid, thanks to the positivity result in
Proposition \ref{prop:spec}. Now we state a result that gives the factorized
system satisfied by the new unknown $(v^\ve_\a)_{1\leq\a\leq N}$.

\begin{lemma}
\label{lem:equiv:cd}
The system (\ref{eq:cd})-(\ref{eq:in}) is equivalent to
\begin{equation}
\label{eq:cd1}
\v^\ve_\a\v^{*\ve}_\a\rho^\ve_\a\frac{\partial v^\ve_\a}{\partial t} +
\frac{1}{\ve}\tilde b^\ve_\a \cdot \nabla v^\ve_\a - \div\left(\tilde D^\ve_\a
\nabla v^\ve_\a\right)+
\frac{1}{\ve^2}\dsp\sum_{\b=1}^N\:\:\Pi^\ve_{\a\b}\v^{*\ve}_\a\v^\ve_\b(v^\ve_\b-v^\ve_\a)
= 0
\end{equation}
in $(0,T)\times\R^d$ for each $1\le\a\le N$ complemented with the initial data:
\begin{equation}
\label{eq:in1}
\dsp v^\ve_\a(0,x) = \frac{u^{in}_\a(x)}{\v^\ve_\a(x)}\hspace{4
cm}x\in\R^d,
\end{equation}
for each $1\le\a\le N$, where the components of $(v^\ve_\a)_{1\le\a\le N}$ are
defined by (\ref{eq:fcp}). The convective velocities, $\tilde b^\ve_\a(x) =
\tilde b_\a\Big(\frac{x}{\ve}\Big)$, in (\ref{eq:cd1}) are given by
\begin{equation}
\label{eq:b1}
\tilde b_\a(y) = \v_\a\v^*_\a b_\a + \v_\a D_\a\ny\v^*_\a -\v^*_\a D_\a
\ny\v_\a\hspace{0.5 cm}\textrm{for every}\:\:1\le\a\le N
\end{equation}
and the diffusion matrices, $\tilde D^\ve_\a(x) = \tilde
D_\a\Big(\frac{x}{\ve}\Big)$, in (\ref{eq:cd1}) are given by
\begin{equation}
\label{eq:D1}
\tilde D_\a(y) = \v_\a\v^*_\a D_\a\hspace{0.5 cm}\textrm{for
every}\:\:1\le\a\le N.
\end{equation}
\end{lemma}

The proof of Lemma \ref{lem:equiv:cd} is just a matter of simple algebra, using 
(\ref{eq:scp:scl}), and we refer to \cite{Ca:02}, \cite{Hu:13} for more details, 
keeping in mind the following chain rule formulae:
$$
\left\{
\begin{array}{l}
\dsp\frac{\partial u^\ve_\a}{\partial t}(t,x) = \exp{(-\lambda t/\ve^2)}
\left(
\frac{-\lambda}{\ve^2}\v_\a\Big(\frac{x}{\ve}\Big)v^\ve_\a(t,x)+\v_\a\Big(\frac{x}{\ve}\Big)\frac{\partial
v^\ve_\a}{\partial t}(t,x)\right),\\[0.5 cm]
\dsp\nabla\Big(u^\ve_\a(t,x)\Big) = \exp{(-\lambda t/\ve^2)}
\left(
\frac{1}{\ve}v^\ve_\a(t,x)\Big(\ny\v_\a\Big)\Big(\frac{x}{\ve}\Big)+\v_\a\Big(\frac{x}{\ve}\Big)\nx
v^\ve_\a(t,x) \right) .
\end{array}\right.
$$

\begin{remark}
\label{rem:div}
The divergence of the convective fields $\tilde b_\a$ satisfy
\begin{equation}
\label{eq:divb}
\div_y\tilde b_\a =
\dsp\sum_{\b=1}^N\:\:\Pi^*_{\a\b}\v_\a\v^*_\b-\dsp\sum_{\b=1}^N\:\:\Pi_{\a\b}\v^*_\a\v_\b.
\end{equation}
It follows that
$$
\sum_{\a=1}^N\:\:\div_y\tilde b_\a=0.
$$
\end{remark}

\begin{remark}
\label{rem:diff}
The factorized system (\ref{eq:cd1}) still has large lower order terms. But, as
noticed in \cite{AC:00, Ca:02}, the terms are transformed as ``difference
terms''. This factorization is the key for getting a priori estimate on the
differences $(v^\ve_\a-v^\ve_\b)$.
\end{remark}

The following lemma gives the a priori estimates on the new unknown.
\begin{lemma}
\label{lem:apriori}
Let $(v^\ve_\a)_{1\le\a\le N}$ be a weak solution of
(\ref{eq:cd1})-(\ref{eq:in1}). There exists a constant $C$, independent of
$\ve$, such that

\begin{equation}
\label{eq:apriori}
\begin{array}{ll}
\dsp\sum_{\a=1}^N\:\:\Big\|v^\ve_\a\Big\|_{L^\infty((0,T);L^2(\R^d))}+\sum_{\a=1}^N\:\:\Big\|\nabla
v^\ve_\a\Big\|_{L^2((0,T)\times\R^d)}\\[0.4 cm]
+\dsp\frac1\ve\sum_{\a=1}^N\:\:\sum_{\b=1}^N\:\:\Big\|v^\ve_\a-v^\ve_\b\Big\|_{L^2((0,T)\times\R^d)}\leq
C\sum_{\a=1}^N\:\:\|v^{in}_\a\|_{L^2(\R^d)}.
\end{array}
\end{equation}
\end{lemma}

\begin{proof}
To derive the a priori estimates, we multiply (\ref{eq:cd1}) by $v^\ve_\a$
followed by integrating over $\R^d$ and sum the obtained expressions over
$1\le\a\le N$:
\begin{equation}
\label{eq:1:nrj}
\begin{array}{cc}
\dsp\frac12\frac{d}{dt}\dsp\sum_{\a=1}^N\int_{\R^d}\v^\ve_\a\v^{*\ve}_\a\rho^\ve_\a|v^\ve_\a|^2{\rm
d}x
-\frac{1}{2\ve}\dsp\sum_{\a=1}^N\int_{\R^d}\div(\tilde
b^\ve_\a)|v^\ve_\a|^2{\rm d}x\\[0.3 cm]
+\dsp\sum_{\a=1}^N\int_{\R^d}\tilde D^\ve_\a \nabla v^\ve_\a\cdot\nabla
v^\ve_\a{\rm d}x
+\frac{1}{\ve^2}\dsp\sum_{\a=1}^N\:\:\sum_{\b=1}^N\int_{\R^d}\Pi^\ve_{\a\b}\v^{*\ve}_\a\v^\ve_\b(v^\ve_\b-v^\ve_\a)v^\ve_\a=0.
\end{array}
\end{equation}
To simplify the above expressions, we now use the scaled spectral problems 
(\ref{eq:scp:scl})-(\ref{eq:ascp:scl}). Multiply
(\ref{eq:scp:scl}) by $\v^{*\ve}_\a(v^\ve_\a)^2$ followed by integration over the space domain $\R^d$:
$$
\frac1\ve\iR(b^\ve_\a \cdot \nabla \v^\ve_\a)\v^{*\ve}_\a(v^\ve_\a)^2{\rm d}x - \iR\div
\Big(D^\ve_\a \nabla \v^\ve_\a\Big)\v^{*\ve}_\a(v^\ve_\a)^2{\rm d}x
$$
$$
+\frac{1}{\ve^2} \sum_{\b=1}^N \iR\Pi^\ve_{\a\b} \v^\ve_\b\v^{*\ve}_\a(v^\ve_\a)^2{\rm d}x -
\frac{1}{\ve^2}\iR\lambda\rho^\ve_\a \v^\ve_\a\v^{*\ve}_\a(v^\ve_\a)^2{\rm d}x
$$
$$
= - \frac1\ve\iR\div(b^\ve_\a\v^{*\ve}_\a)\v^\ve_\a(v^\ve_\a)^2{\rm d}x - \frac1\ve\iR\v^\ve_\a\v^{*\ve}_\a
b^\ve_\a\cdot\nabla(v^\ve_\a)^2{\rm d}x
$$
$$ + \iR \v^{*\ve}_\a D^\ve_\a \nabla
\v^\ve_\a\cdot\nabla(v^\ve_\a)^2{\rm d}x
+ \iR(v^\ve_\a)^2  D^\ve_\a \nabla \v^\ve_\a\cdot\nabla\v^{*\ve}_\a{\rm d}x 
$$
$$
+\dsp\frac{1}{\ve^2}\sum_{\b=1}^N\:\iR \Pi^\ve_{\a\b}\v^{*\ve}_\a\v^\ve_\b(v^\ve_\a)^2{\rm d}x - \frac{1}{\ve^2}\iR
\lambda\rho^\ve_\a\v^\ve_\a\v^{*\ve}_\a(v^\ve_\a)^2{\rm d}x
$$
$$
= - \frac1\ve\iR\div(b^\ve_\a\v^{*\ve}_\a)\v^\ve_\a(v^\ve_\a)^2{\rm d}x - \frac1\ve\iR\v^\ve_\a\v^{*\ve}_\a
b^\ve_\a\cdot\nabla(v^\ve_\a)^2{\rm d}x
$$
$$
+ \iR \v^{*\ve}_\a D^\ve_\a \nabla
\v^\ve_\a\cdot\nabla(v^\ve_\a)^2{\rm d}x
+ \iR(v^\ve_\a)^2  D^\ve_\a \nabla \v^\ve_\a\cdot\nabla\v^{*\ve}_\a{\rm d}x 
$$
$$
- \iR \v^\ve_\a D^\ve_\a \nabla \v^{*\ve}_\a\cdot\nabla(v^\ve_\a)^2{\rm d}x + \iR \v^\ve_\a D^\ve_\a \nabla
\v^{*\ve}_\a\cdot\nabla(v^\ve_\a)^2{\rm d}x
$$
$$
+
\dsp\frac{1}{\ve^2}\sum_{\b=1}^N\:\iR \Pi^\ve_{\a\b}\v^{*\ve}_\a\v^\ve_\b(v^\ve_\a)^2{\rm d}x - \frac{1}{\ve^2}\iR
\lambda\rho^\ve_\a\v^\ve_\a\v^{*\ve}_\a(v^\ve_\a)^2{\rm d}x
$$
$$
= - \frac1\ve\iR\div(b^\ve_\a\v^{*\ve}_\a)\v^\ve_\a(v^\ve_\a)^2{\rm d}x - \iR \div(D^\ve_\a \nabla
\v^{*\ve}_\a)\v^\ve_\a (v^\ve_\a)^2{\rm d}x 
$$
$$
- \frac{1}{\ve^2}\iR
\lambda\rho^\ve_\a\v^\ve_\a\v^{*\ve}_\a(v^\ve_\a)^2{\rm d}x
- \frac1\ve\iR\v^\ve_\a\v^{*\ve}_\a b^\ve_\a\cdot\nabla(v^\ve_\a)^2{\rm d}x 
$$
$$
+ \iR \v^{*\ve}_\a D^\ve_\a \nabla
\v^\ve_\a\cdot\nabla(v^\ve_\a)^2{\rm d}x - \iR \v^\ve_\a D^\ve_\a \nabla
\v^{*\ve}_\a\cdot\nabla(v^\ve_\a)^2{\rm d}x
$$
$$
+ \dsp\frac{1}{\ve^2}\sum_{\b=1}^N\:\iR \Pi^\ve_{\a\b}\v^{*\ve}_\a\v^\ve_\b(v^\ve_\a)^2{\rm d}x = 0.
$$
In the above expression, we recognize the scaled adjoint cell problem (\ref{eq:ascp:scl}).
We also recognize the scaled expression of (\ref{eq:b1}) for the convective field $\tilde
b_\a$. Taking all these into consideration, we have the following:
$$
- \frac1\ve\iR\tilde b^\ve_\a\cdot\nabla(v^\ve_\a)^2{\rm d}x + \dsp\frac{1}{\ve^2}\sum_{\b=1}^N\:\iR
\Big(\Pi^\ve_{\a\b}\v^{*\ve}_\a\v^\ve_\b - \Pi^{*\ve}_{\a\b}\v^{*\ve}_\b\v^\ve_\a\Big)(v^\ve_\a)^2{\rm d}x=0.
$$
Summing over $\a$, we have:
\begin{equation}
\label{eq:divb:nrj}
\dsp-\frac1{2\ve}\sum_{\a=1}^N\,\,\iR\div(\tilde b^\ve_\a) (v^\ve_\a)^2{\rm d}x = \dsp\frac{1}{2\ve^2}\sum_{\a=1}^N\sum_{\b=1}^N\:\iR
\Big(\Pi^\ve_{\a\b}\v^{*\ve}_\a\v^\ve_\b - \Pi^{*\ve}_{\a\b}\v^{*\ve}_\b\v^\ve_\a\Big)(v^\ve_\a)^2{\rm d}x.
\end{equation}
Now, let us employ (\ref{eq:divb:nrj}) in the estimate (\ref{eq:1:nrj}) which leads to:
$$
\dsp\frac12\frac{d}{dt}\dsp\sum_{\a=1}^N\int_{\R^d}\v^\ve_\a\v^{*\ve}_\a\rho^\ve_\a|v^\ve_\a|^2{\rm
d}x
+\dsp\sum_{\a=1}^N\int_{\R^d}\tilde D^\ve_\a \nabla v^\ve_\a\cdot\nabla
v^\ve_\a{\rm d}x
$$
$$
+\dsp\frac{1}{\ve^2}\dsp\sum_{\a=1}^N\:\:\sum_{\b=1}^N\iR\Big\{\Pi^\ve_{\a\b}\v^{*\ve}_\a\v^\ve_\b\Big(v^\ve_\b v^\ve_\a - \frac12 (v^\ve_\a)^2\Big) - \frac12\Pi^{*\ve}_{\a\b}\v^{*\ve}_\b\v^\ve_\a (v^\ve_\a)^2 \Big\}{\rm d}x=0.
$$
The above expression is nothing but the following energy estimate:
\begin{equation}
\label{eq:nrj}
\begin{array}{ll}
\dsp\frac12\frac{d}{dt}\dsp\sum_{\a=1}^N\int_{\R^d}\v^\ve_\a\v^{*\ve}_\a\rho^\ve_\a|v^\ve_\a|^2{\rm
d}x
+\dsp\sum_{\a=1}^N\int_{\R^d}\tilde D^\ve_\a \nabla v^\ve_\a\cdot\nabla
v^\ve_\a{\rm d}x\\[0.3 cm]
-\dsp\frac{1}{2\ve^2}\dsp\sum_{\a=1}^N\:\:\sum_{\b=1}^N\int_{\R^d}\Pi^\ve_{\a\b}\v^{*\ve}_\a\v^\ve_\b|v^\ve_\a-v^\ve_\b|^2{\rm
d}x=0.
\end{array}
\end{equation}
Each one of the integrands in the above estimate is positive because of the positivity assumption (\ref{eq:ass:rho}), coercivity assumption (\ref{eq:ass:diff}) and the cooperative assumption (\ref{eq:ass:cple}). Integrating the energy estimate (\ref{eq:nrj}) over $(0,T)$ yields the a priori estimates (\ref{eq:apriori}).
\end{proof}

\section{Two-scale Compactness}
\label{sec:2scl}

The homogenization procedure is to consider the weak formulation of
(\ref{eq:cd1})-(\ref{eq:in1}) with appropriately chosen test functions and
passing to the limit as $\ve\to0$. The usual approach is to obtain two-scale
limits using a priori estimates of Lemma \ref{lem:apriori} by employing some
compactness theorems. As it has been noticed in \cite{MP:05, DP:05, AR:07}, the
classical notion of two-scale convergence from \cite{Al:92b, nguetseng}
needs to be modified in order to address the homogenization of parabolic
problems in strong convection regime. We recall this modified notion of
two-scale convergence with drift, as first defined in \cite{MP:05}.

\begin{definition}
\label{def:2scl}
Let $b^*\in\R^d$ be a constant vector. A sequence of functions $u_\ve(t,x)$ in
$L^2((0,T)\times\R^d)$ is said to two-scale converge with drift $b^*$, or
equivalently in moving coordinates $\dsp(t,x)\rightarrow\Big(t,x-\frac{b^*
t}{\ve}\Big)$, to a limit $u_0(t,x,y)\in L^2((0,T)\times\R^d\times Y)$ if, for
any function $\phi(t,x,y)\in C^\infty_c((0,T)\times\R^d;C^\infty_\#(Y))$, we
have
\begin{equation}
\label{eq:hom:defn2d}
\lim_{\ve\to0}\int_0^T\int_{\R^d}
u_\ve(t,x)\phi\Big(t,x-\frac{b^*}{\ve}t,\frac{x}{\ve}\Big){\rm d}x{\rm d}t =
\int_0^T\int_{\R^d}\iy u_0(t,x,y)\phi(t,x,y){\rm d}y{\rm d}x{\rm d}t.
\end{equation}
We denote this convergence by $u_\ve \ts u_0$.
\end{definition}

Now we state a compactness theorem, again borrowed from \cite{MP:05}, which
guarantees the existence of two-scale limits with drift for certain sequences.

\begin{proposition}
\label{prop:2scl1}\cite{MP:05, Al:08}
Let $b^*$ be a constant vector in $\mathbb{R}^d$ and let the sequence $u_\ve$
be uniformly bounded in $L^2((0,T)\times\R^d)$.
Then, there exist a subsequence, still denoted by $\ve$, and a function
$u_0(t,x,y)\in L^2((0,T)\times\R^d;L^2_\#(Y))$such that
$$
u_\ve \ts u_0.
$$
\end{proposition}

\begin{remark}
\label{rem:2scl}
Note that the case $b^*=0$ coincides with the classical notion of
two-scale convergence from \cite{Al:92b, nguetseng}. It should also be noted that the
two-scale limits obtained according to Proposition \ref{prop:2scl1} depend on
the chosen drift velocity $b^*\in\R^d$. These issues are addressed in
\cite{Hu:13}. Unfortunately, the notion of convergence in
Definition \ref{def:2scl} does not carry over to the case when the drift
velocity $b^*$ varies in space.
\end{remark}

If the sequence $\{u_\ve\}$ has additional bounds, then the result of
Proposition \ref{prop:2scl1} can be improved. The following result addresses
this issue when the sequence has uniform $H^1$ bounds in space.
\begin{proposition}
\label{prop:2scl}\cite{MP:05, Al:08}
Let $b^*$ be a constant vector in $\mathbb{R}^d$ and let the sequence $u_\ve$
be uniformly bounded in $L^2((0,T);H^1(\mathbb{R}^d))$.
Then, there exist a subsequence, still denoted by $\ve$, and functions
$u_0(t,x) \in L^2((0,T);H^1(\mathbb{R}^d))$ and $u_1(t,x,y) \in
 L^2((0,T)\times\mathbb{R}^d;H^1_\#(Y))$ such that
$$
u_\ve \ts u_0
$$
and
$$
\nabla u_\ve \ts \nabla_x u_0 + \nabla_y u_1.
$$
\end{proposition}

Having given the notion of convergence, we shall state a result that gives the
two-scale limits corresponding to solution sequences for
(\ref{eq:cd1})-(\ref{eq:in1}).

\begin{theorem}
\label{thm:3:2scl}
Let $b^*\in\R^d$ be a constant vector. There exist $v\in L^2((0,T);H^1(\R^d))$ and
$v_{1,\a}\in L^2((0,T)\times\R^d;H^1_\#(Y))$, for each $1\le\a\le N$, such that
a subsequence of solutions $(v^\ve_\a)_{1\le\a\le N}\in L^2((0,T);H^1(\R^d))^N$
of the system (\ref{eq:cd1})-(\ref{eq:in1}), two-scale converge with drift
$b^*$, as $\ve \to 0$, in the following sense:
\begin{equation}
\label{eq:2scl}
\begin{array}{cc}
\dsp v^\ve_\a \ts v, \hspace{1 cm} \nabla v^\ve_\a \ts \nx v + \ny
v_{1,\a},\\[0.2 cm]
\dsp\frac{1}{\ve}\Big(v^\ve_\a-v^\ve_\b\Big) \ts v_{1,\a} - v_{1,\b},
\end{array}
\end{equation}
for every $1\le\a,\b\le N$.
\end{theorem}

\begin{proof}
Consider the a priori bounds (\ref{eq:apriori}) on $v^\ve_\a$ obtained in Lemma
\ref{lem:apriori}. It follows from Proposition \ref{prop:2scl} that there exist
a subsequence (still indexed by $\ve$) and two-scale limits, say $v_\a\in
L^2((0,T);H^1(\mathbb{R}^d))$ and $v_{1,\a}\in L^2((0,T)\times\R^d;H^1_\#(Y))$
such that
\begin{equation}
\label{eq:2scl:1}
\begin{array}{cc}
v^\ve_\a\ts v_\a\\[0.2 cm]
\nabla v^\ve_\a \ts \nx v_\a + \ny v_{1,\a}
\end{array}
\end{equation}
for every $1\le\a\le N$. Also from the a priori estimates (\ref{eq:apriori}) we
have:
\begin{equation}
\label{eq:apri:mod}
\dsp\sum_{\a=1}^N\:\:\sum_{\b=1}^N\:\:\int_0^T\int_{\R^d}
\Big(v^\ve_\a-v^\ve_\b\Big)^2{\rm d}x{\rm d}t\leq C\ve^2.
\end{equation}
The estimate (\ref{eq:apri:mod}) implies that the two-scale limits obtained in
the first line of (\ref{eq:2scl:1}) do match i.e., $v_\a=v$ for every
$1\le\a\le N$. However, the limit of the coupled term isn't straightforward.
Since $\dsp\frac1\ve(v^\ve_\a - v^\ve_\b)$ is bounded in
$L^2((0,T)\times\R^d)$, we have the existence of a subsequence and a function
$q(t,x,y)\in L^2((0,T)\times\R^d;L^2_\#(Y))$ from Proposition \ref{prop:2scl1}
such that
\begin{equation}
\label{eq:2scl:2}
\frac1\ve(v^\ve_\a - v^\ve_\b) \ts q(t,x,y) .
\end{equation}
Taking $\Psi\in L^2((0,T)\times\R^d\times Y)^d$, let us consider
\begin{equation}
\label{eq:2scl:3}
\begin{array}{cc}
\dsp\int_0^T\int_{\R^d}\Big(\nabla v^\ve_\a - \nabla
v^\ve_\b\Big)\cdot\Psi\Big(t,x-\frac{b^*}{\ve}t,\frac{x}{\ve}\Big){\rm d}x{\rm
d}t=\\[0.3 cm]
\dsp-\int_0^T\int_{\R^d}\Big(v^\ve_\a -
v^\ve_\b\Big)\div_x\Psi\Big(t,x-\frac{b^*}{\ve}t,\frac{x}{\ve}\Big){\rm d}x{\rm
d}t\\[0.3 cm]
\dsp-\int_0^T\int_{\R^d}\frac{1}{\ve}\Big(v^\ve_\a -
v^\ve_\b\Big)\div_y\Psi\Big(t,x-\frac{b^*}{\ve}t,\frac{x}{\ve}\Big){\rm d}x{\rm
d}t.
\end{array}
\end{equation}
Let us pass to the limit in (\ref{eq:2scl:3}) as $\ve\to0$. The first term on
the right hand side vanishes as the limits of $v^\ve_\a$ and $v^\ve_\b$ match.
To pass to the limit in the second term of the right hand side, we shall use
(\ref{eq:2scl:2}). Considering the two-scale limit in the second line of
(\ref{eq:2scl:1}), upon passing to the limit as $\ve\to0$ in (\ref{eq:2scl:3})
we have:
\begin{equation}
\label{eq:2scl:4}
\begin{array}{cc}
\dsp\int_0^T\int_{\R^d}\iy\nabla_y\Big(v_{1,\a} -
v_{1,\b}\Big)\cdot\Psi(t,x,y){\rm d}y{\rm d}x{\rm d}t =\\[0.3 cm]
-\dsp\int_0^T\int_{\R^d}\iy q(t,x,y)\div_y\Psi(t,x,y){\rm d}y{\rm d}x{\rm d}t.
\end{array}
\end{equation}
>From (\ref{eq:2scl:4}) we deduce that $(v_{1,\a} - v_{1,\b})$ and $q(t,x,y)$
differ by a function of $(t,x)$, say $l(t,x)$. As $v_{1,\a}$ and $v_{1,\b}$ are
also defined up to the addition of a function solely dependent on $(t,x)$, we
can get rid of $l(t,x)$ and we recover indeed the following limit $q(t,x,y) =
v_{1,\a} - v_{1,\b}$.
\end{proof}

\section{Homogenization Result}
\label{sec:hom}
This section deals with the homogenization of the coupled system
(\ref{eq:cd1})-(\ref{eq:in1}). To begin with, we state a Fredholm alternative 
for solving the cell problem, which is a key ingredient in the homogenization result.

\begin{lemma}
\label{lem:fh}
Let $(f_\a)_{1\le\a\le N}\in(L^2_\#(Y))^N$. Consider the following cooperative
system:
\begin{equation}
\label{eq:fh}
\left\{
\begin{array}{l}
\tilde b_\a\cdot \ny\zeta_\a - \div_y \Big(\tilde D_\a \nabla_y\zeta_\a\Big)
+ \dsp\sum_{\b=1}^N\: \Pi_{\a\b} \v^*_\a\v_\b\Big(\zeta_\b-\zeta_\a\Big)= f_\a
\mbox{ in }Y,\\[0.3cm]
y \to \zeta_\a(y) \quad Y\mbox{-periodic},
\end{array} \right.
\end{equation}
for every $1\le\a\le N$, where the coefficients $(\tilde b_\a, \tilde D_\a)$
are as in (\ref{eq:b1})-(\ref{eq:D1}) and the hypotheses
(\ref{eq:ass:vel})-(\ref{eq:ass:irr}) hold. Then there exists a unique solution
$(\zeta_\a)_{1\le\a\le N}\in(H^1_\#(Y))^N/(\R\times\mathds{1})$ to (\ref{eq:fh}), where $\mathds{1}=(1,\cdots,1)\in\R^N$, if and only if the
following compatibility condition holds true:
\begin{equation}
\label{eq:fhc}
\sum_{\a=1}^N\iy f_\a{\rm d}y=0.
\end{equation}
\end{lemma}

\begin{proof}
To prove that condition (\ref{eq:fhc}) is necessary, let us integrate the left hand side of (\ref{eq:fh}) over the unit cell. Exploiting the periodic boundary conditions, we will be left with:
$$
-\iy\div_y(\tilde b_\a)\zeta_\a{\rm d}y + \dsp\sum_{\b=1}^N\iy \Pi_{\a\b} \v^*_\a\v_\b\Big(\zeta_\b-\zeta_\a\Big){\rm d}y.
$$
Substituting for the divergence term in the above expression from (\ref{eq:divb}) and summing over $\a$ indeed guarantees that the condition (\ref{eq:fhc}) on the source term is necessary.

To prove sufficiency, let us assume that (\ref{eq:fhc}) is satisfied. Consider the following norm on the quotient space $\mathscr{H}(Y) :=(H^1_\#(Y))^N/(\R\times\mathds{1})$:
\begin{equation}
\label{eq:quo:norm}
\|(z_\a)_{1\le\a\le N}\|^2_{\mathscr{H}(Y)}=\sum_{\a=1}^N\|\ny z_\a\|^2_{L^2(Y)} + \sum_{\a=1}^N\sum_{\b=1}^N\|z_\a-z_\b\|^2_{L^2(Y)}.
\end{equation}
(It is easy to show that (\ref{eq:quo:norm}) is a norm on $\mathscr{H}(Y)$ since 
the zero set of (\ref{eq:quo:norm}) is the subspace spanned by $\mathds{1}$.)  
The variational formulation of (\ref{eq:fh}) in $\mathscr{H}(Y)$ is: find 
$\zeta=(\zeta_a)_{1\le\a\le N} \in \mathscr{H}(Y)$ such that
\begin{equation}
\label{eq:fh:vf}
\iy Q(\zeta)\cdot\eta\:{\rm d}y = L(\eta) \mbox{ for any }\eta=(\eta_a)_{1\le\a\le N} \in \mathscr{H}(Y),
\end{equation}
with
$$
\begin{array}{cc}
\dsp \iy Q(\zeta)\cdot\eta{\rm d}y := \sum_{\a=1}^N\iy\Big(\tilde b_\a(y)\cdot \ny\zeta_\a\Big)\eta_\a{\rm d}y +
\sum_{\a=1}^N\iy\tilde D_\a(y)\ny\zeta_\a\cdot\nabla_y\eta_\a{\rm d}y\\[0.2 cm]
\dsp+\sum_{\a=1}^N\sum_{\b=1}^N\iy\Pi_{\a\b}
\v^*_\a\v_\b\Big(\zeta_\b-\zeta_\a\Big)\eta_\a{\rm d}y
\end{array}
$$
and
$$
L(\eta):=\sum_{\a=1}^N\iy
f_\a\eta_\a{\rm d}y.
$$
The compatibility condition (\ref{eq:fhc}) implies that $(f_\a)_{1\le\a\le N}$ is 
orthogonal to $\mathds{1}$ in $L^2$ and consequently that the linear form $L(\eta)$ 
in (\ref{eq:fh:vf}) is continuous.

By performing similar computations as in the proof of Lemma \ref{lem:apriori},
we can show that the bilinear form in (\ref{eq:fh:vf}) is coercive in $\mathscr{H}(Y)$ i.e.,
$$
\iy Q(\zeta)\cdot\zeta\:{\rm d}y \ge C\sum_{\a=1}^N\iy |\nabla_y \zeta_\a|^2{\rm d}y
+ \sum_{\a=1}^N\sum_{\b=1}^N\iy |\zeta_a - \zeta_b|^2{\rm d}y.
$$
To show that the bilinear form in (\ref{eq:fh:vf}) is continuous on $\mathscr{H}(Y)\times\mathscr{H}(Y)$ 
we remark that, first, $\iy Q(\eta)\cdot\mathds{1}\,{\rm d}y=0$ for any $\eta\in \mathscr{H}(Y)$ (this is 
precisely the computation which yields the compatibility condition (\ref{eq:fhc})) and, second, 
$Q(\eta - c \mathds{1}) = 0$ for any $\eta\in \mathscr{H}(Y)$ and any $c\in\R$. Therefore, for any 
$\zeta,\eta \in \mathscr{H}(Y)$, we have the following:
$$
\iy Q(\zeta)\cdot\eta\:{\rm d}y = \iy Q\Big(\zeta - \mathds{1}c_\zeta\Big) \cdot 
\Big(\eta - \mathds{1}c_\eta\Big)\:{\rm d}y\hspace{0.5 cm}\mbox{ for any constants } c_\zeta, c_\eta \in\R,
$$
which implies
$$
\Big|\iy Q(\zeta)\cdot\eta\:{\rm d}y\Big|\le C\Big\|\Big(\zeta - \mathds{1}c_\zeta\Big)\Big\|_{(H^1_\#(Y))^N}\Big\|\Big(\eta - \mathds{1}c_\eta\Big)\Big\|_{(H^1_\#(Y))^N}=C\|\zeta\|_{\mathscr{H}(Y)}\|\eta\|_{\mathscr{H}(Y)}.
$$
We can thus apply the Lax-Milgram lemma in $\mathscr{H}(Y)$ to obtain the existence and uniqueness of a solution to (\ref{eq:fh}).
\end{proof}

\begin{remark}
\label{rem:quot}
The well-posedness result of the Lemma \ref{lem:fh} is given in the quotient space $(H^1_\#(Y))^N/(\R\times\mathds{1})$ i.e., the solutions are unique up to the addition of a constant. The constant being the same for each component of the solution.
\end{remark}

In the previous section, using the a priori estimates, we have obtained
two-scale limits with drift for the solution sequence. Now, by choosing an 
appropriate drift constant $b^*$, we shall characterize the two scale limits.
Contrary to the compactness result of Theorem \ref{thm:3:2scl} which gives the
convergence up to a subsequence, the next result guarantees that the entire
sequence $v^\ve_\a$ converges to $v$ for every $1\le\a\le N$. 
The main result of this article is the following.

\begin{theorem}
\label{thm:hom}
Let $(v^\ve_\a)_{1\le\a\le N}$ be the sequence of solutions to the system
(\ref{eq:cd1})-(\ref{eq:in1}). The entire sequence $v^\ve_\a$ converges, 
in the sense of Theorem \ref{thm:3:2scl}, to the limits $v\in L^2((0,T);H^1(\R^d))$ 
and $v_{1,\a}\in L^2((0,T)\times\R^d;H^1_\#(Y))$ for every $1\le\a\le N$ 
(see (\ref{eq:2scl}) for details).
The two-scale limits $v_{1,\a}$ are explicitly given by
\begin{equation}
\label{eq:2scl:corr}
v_{1,\a}(t,x,y) = \sum_{i=1}^d \frac{\partial v}{\partial
x_i}(t,x)\omega_{i,\a}(y)\hspace{0.5 cm}\mbox{ for every }1\le\a\le N,
\end{equation}
where $(\omega_{i,\a})_{1\le\a\le N}\in(H^1_\#(Y))^N/(\R\times\mathds{1})$ satisfy the cell
problem:
\begin{equation}
\label{eq:cpb}
\left\{
\begin{array}{lll}
\tilde b_\a(y)\cdot \Big(\ny\omega_{i,\a} + e_i\Big) - \div_y \Big(\tilde D_\a
\Big(\nabla_y\omega_{i,\a}+e_i\Big)\Big)\\[0.3cm]
+ \dsp\sum_{\b=1}^N\: \Pi_{\a\b}
\v^*_\a\v_\b\Big(\omega_{i,\b}-\omega_{i,\a}\Big)= \v_\a\v^*_\a\rho_\a b^*\cdot
e_i & \mbox{ in }Y,\\[0.3cm]
y \to \omega_{i,\a}(y) & Y\mbox{-periodic},
\end{array} \right.
\end{equation}
for every $1\le i\le d$, where the drift velocity $b^*$ is given by
\begin{equation}
\label{eq:drift}
\dsp b^* = \dsp\sum_{\a=1}^N\: \iy \tilde b_\a(y){\rm d}y.
\end{equation}
Further, the two-scale limit $v(t,x)$ is the unique solution of the scalar
diffusion equation:
\begin{equation}
\label{eq:hom}
\left\{
\begin{array}{ll}
\dsp\frac{\partial v}{\partial t} - \div({\mathcal D}\nabla v) = 0 &
\mbox{ in }(0,T)\times\R^d,\\[0.3 cm]
v(0,x) = \dsp\sum_{\a=1}^N u^{in}_\a(x) \dsp\iy\rho_\a(y) \v^*_\a(y){\rm d}y & \mbox{ in
}\R^d,
\end{array} \right.
\end{equation}
with the elements of the dispersion matrix $\mathcal D$ given by
\begin{equation}
\label{eq:disp}
\begin{array}{ll}
{\mathcal D}_{ij}=\dsp\sum_{\a=1}^N\:\:\iy\tilde D_\a\Big(\nabla_y\omega_{i,\a}
+ e_i\Big)\cdot\Big(\nabla_y\omega_{j,\a} + e_j\Big){\rm d}y\\[0.5 cm]
-\dsp\frac12\sum_{\a,\b=1}^N\:\:\iy\v^*_\a\v_\b\Pi_{\a\b}\Big(\omega_{i,\a} -
\omega_{i,\b}\Big)\Big(\omega_{j,\a} - \omega_{j,\b}\Big){\rm d}y.
\end{array}
\end{equation}
\end{theorem}

\begin{remark}
\label{rem:irr}
The irreducibility assumption (\ref{eq:ass:irr}) on the coupling matrix $\Pi$
ensures microscopic equilibrium among all $v^\ve_\a$ resulting in a single
homogenized limit $v(t,x)$ i.e., if the coupling matrix $\Pi\equiv0$ (say), we
get $N$ different homogenized limits.
\end{remark}

\begin{remark}
Our main homogenization result (Theorem \ref{thm:hom}) holds only for weakly
coupled cooperative parabolic systems. Our approach does not answer the
homogenization of general weakly coupled parabolic systems, not to mention
fully coupled systems. We heavily rely upon the cooperative assumption on the
coupling matrix as the positivity and spectral theorems are known only in the
cooperative case.
\end{remark}

\begin{remark}
\label{rem:lim}
The homogenized limit $v(t,x)$ is proven to satisfy a scalar diffusion equation
(\ref{eq:hom}), which is a bit deceptive by its simplicity. However, if we make 
the following change of functions:
$$
\tilde v(t,x) = \exp{(-\lambda t/\ve^2)}v\Big(t,x-\frac{b^*}{\ve}t\Big),
$$
we remark that $\tilde v(t,x)$ indeed satisfies the following scalar
convection-diffusion-reaction equation:
$$
\dsp\frac{\partial \tilde v}{\partial t} +
\frac{b^*}{\ve}\cdot\nabla \tilde v - \div({\mathcal D}\nabla \tilde v) +
\frac{\lambda}{\ve^2}\tilde v = 0\mbox{ in }(0,T)\times\R^d.
$$
Therefore, $b^*/\ve$ is precisely the effective drift while $\lambda/\ve^2$ 
is the effective reaction rate. 
Remark that because of the large drift $\ve^{-1}b^*$, we cannot work in
bounded domains.
\end{remark}

\begin{remark}
The assumption of pure periodicity on the coefficients of (\ref{eq:cd}) is 
crucial for the results obtained in this article. The natural thought for
generalizing the results of this article is to explore the possibility of
considering ``locally periodic'' coefficients i.e., coefficients of the type
$b(x,x/\ve)$, where the function is $Y$-periodic in the second variable. If
the convective fields $b^\ve_\a$ were locally periodic, then it is clear 
that the drift vector $b^*(x)$ should depend on $x$. However, in such a 
case, we have no idea on how to extend the method of two-scale asymptotic 
expansion, not to mention the even greater difficulties in generalizing 
the notion of two-scale convergence with non-constant drift (as already 
mentioned in Remark \ref{rem:2scl}). Such a generalization still
remains as an outstanding open problem in the theory of Taylor dispersion. 
\end{remark}

\begin{remark}
This article only addresses the homogenization of linear systems. We have also
considered only diagonal diffusion models. Cross diffusion phenomena occurs
naturally in the physics of multicomponent gaseous mixtures, population
dynamics and porous media (cf. \cite{BGS:12} and references therein). The
natural nonlinear transport model to consider is the Maxwell-Stefan's
equations. A complete mathematical study of the Maxwell-Stefan laws is still
missing. There have been some recent studies in this direction (cf.
\cite{BGS:12, DLM:14, MT:14, DT:p} for example). One approach would be to
consider the ``parabolically'' scaled Maxwell-Stefan's equations and arrive at
an homogenization result. The obvious questions to ask is the following: Is
there a scalar diffusion limit even in case of nonlinear Maxwell-Stefan's
equations? This problem might involve mathematical techniques quite different
from the ones used here as the spectral problems (which is the crux of the
Factorization method) in the nonlinear counterpart have not been well
understood. We hope to return to this question in subsequent publications.
\end{remark}

Before we present the proof of Theorem \ref{thm:hom}, we state a lemma that
gives some qualitative information on the dispersion matrix.

\begin{lemma}
\label{lem:disp}
The dispersion matrix $\mathcal D$ given by (\ref{eq:disp}) is symmetric
positive definite.
\end{lemma}

\begin{proof}
The symmetric part is obvious. By the hypothesis on the coupling matrix $\Pi$
and the positivity of the first eigenvector functions, the factor
$\Pi_{\a\b}\v^*_\a\v_\b$ is always non-positive for $\alpha\neq\beta$. By the hypothesis
(\ref{eq:ass:diff}), we know that the diffusion matrices $D_\a$ are coercive 
with coercivity constants $c_\a>0$.
For $\xi\in\R^d$, we define
$$
\omega_{\a\xi} := \sum_{i=1}^d\omega_{i,\a}\xi_i.
$$
Then,
$$
{\mathcal D}\xi\cdot\xi
\geq\dsp\sum_{\a=1}^N\:\:c_\a\dsp\iy\Big|\nabla_y\omega_{\a\xi}+\xi\Big|^2{\rm
d}y \geq0.
$$
Now, we need to show that ${\mathcal D}\xi\cdot\xi>0$ for all $\xi\not=0$.
Suppose that ${\mathcal D}\xi\cdot\xi=0$ which in turn implies that
$\omega_{\a\xi}+\xi\cdot y\equiv C_\a$ for some constant $C_\a$. As the cell solutions
$(\omega_{i,\a})_{1\le\a\le N}$ are $Y$-periodic, they cannot be affine. Thus the
above equalities are possible only when $\xi=0$ which implies the positive
definiteness of $\mathcal D$.
\end{proof}

\begin{proof}[Proof of Theorem \ref{thm:hom}]
In the sequel we use the notations
$$
\begin{array}{cc}
\dsp\phi\equiv\phi(t,x) \, , \quad \phi^\ve\equiv\phi\left(t,x-\frac{b^*
t}{\ve}\right) \, ,\\[0.3 cm]
\dsp\phi_{1,\a}\equiv\phi_{1,\a}(t,x,y) \, , \quad \phi^\ve_{1,\a}\equiv\phi_{1,\a}\left(t,x-\frac{b^*
t}{\ve},\frac{x}{\ve}\right).
\end{array}
$$
The idea is to test the factorized equation (\ref{eq:cd1}) with
$$
\phi^\ve_\a=\phi^\ve+\ve\,\phi^\ve_{1,\a},
$$
where $\phi(t,x)$ and $\phi_{1,\a}(t,x,y)$ are smooth functions with compact support in $x$, which
vanish at the final time $T$ and are $Y$-periodic with respect to $y$. We get
$$
\begin{array}{cc}
 \dsp\sum_{\a=1}^N\:\:\iti\iR\v^\ve_\a\v^{*\ve}_\a\rho^\ve_\a\frac{\partial
v^\ve_\a}{\partial t}\phi^\ve_\a{\rm d}x{\rm d}t +
\frac{1}{\ve}\sum_{\a=1}^N\:\:\iti\iR\tilde b^\ve_\a \cdot \nabla
v^\ve_\a\phi^\ve_\a{\rm d}x{\rm d}t\\[0.3 cm]
+\dsp\sum_{\a=1}^N\:\:\iti\iR \tilde D^\ve_\a \nabla v^\ve_\a\cdot \nabla
\phi^\ve_\a{\rm d}x{\rm d}t
+\frac{1}{\ve^2}\dsp\sum_{\a=1}^N\:\:\sum_{\b=1}^N\:\:\iti\iR\Pi^\ve_{\a\b}\v^{*\ve}_\a\v^\ve_\b(v^\ve_\b-v^\ve_\a)\phi^\ve_\a
= 0.
\end{array}
$$
Substituting for $\phi^\ve_\a$ in the above variational formulation and
integrating by parts leads to
\begin{equation}
\label{eq:vf:1}
\begin{array}{cc}
-\dsp\sum_{\a=1}^N\:\:\iti\iR\v^\ve_\a\v^{*\ve}_\a\rho^\ve_\a v^\ve_\a\frac{\partial
\phi^\ve}{\partial t}{\rm d}x{\rm d}t
-\dsp\sum_{\a=1}^N\:\:\iR\v^\ve_\a\v^{*\ve}_\a\rho^\ve_\a v^\ve_\a(0,x)\phi^\ve(0,x){\rm d}x\\[0.3 cm]
\dsp+\frac{1}{\ve}\dsp\sum_{\a=1}^N\:\:\iti\iR v^\ve_\a\Big(\v^\ve_\a\v^{*\ve}_\a\rho^\ve_\a
b^* - \tilde b^\ve_\a\Big)\cdot \nx \phi^\ve{\rm d}x{\rm d}t\\[0.3 cm]
\dsp-\frac{1}{\ve}\dsp\sum_{\a=1}^N\:\:\iti\iR \div\Big(\tilde b^\ve_\a\Big)
v^\ve_\a\phi^\ve{\rm d}x{\rm d}t
+\sum_{\a=1}^N\:\:\iti\iR\tilde D^\ve_\a \nabla v^\ve_\a\cdot\nx\phi^\ve{\rm
d}x{\rm d}t\\[0.3 cm]
+\dsp\frac{1}{\ve^2}\sum_{\a=1}^N\:\:\sum_{\b=1}^N\:\:\iti\iR
\Pi^\ve_{\a\b}\v^{*\ve}_\a\v^\ve_\b(v^\ve_\b-v^\ve_\a)\phi^\ve{\rm d}x{\rm d}t
+ \sum_{\a=1}^N\:\:\iti\iR \Big(\tilde b^\ve_\a \cdot \nabla
v^\ve_\a\Big)\phi_{1,\a}^\ve{\rm d}x{\rm d}t\\[0.3 cm]
+\dsp\sum_{\a=1}^N\:\:\iti\iR\v^\ve_\a\v^{*\ve}_\a\rho^\ve_\a v^\ve_\a b^*\cdot \nabla
\phi_{1,\a}^\ve{\rm d}x{\rm d}t
+ \dsp\sum_{\a=1}^N\:\:\iti\iR \tilde D^\ve_\a \nabla v^\ve_\a\cdot \ny
\phi_{1,\a}^\ve{\rm d}x{\rm d}t\\[0.3 cm]
+ \dsp\frac{1}{\ve}\sum_{\a=1}^N\:\:\sum_{\b=1}^N\:\:\iti\iR
\Pi^\ve_{\a\b}\v^{*\ve}_\a\v^\ve_\b(v^\ve_\b-v^\ve_\a)\phi_{1,\a}^\ve{\rm d}x{\rm d}t
+ {\mathcal O}(\ve) = 0.
\end{array}
\end{equation}

In a first step we choose $\phi^\ve\equiv0$ in (\ref{eq:vf:1}) and pass to the limit as
$\ve\to0$ which yields:
\begin{equation}
\label{eq:vf:cpb}
\begin{array}{cc}
-\dsp\sum_{\a=1}^N\:\:\iti\iR\iy\v_\a\v^*_\a\rho_\a b^*\cdot\nx
v\phi_{1,\a}{\rm d}y{\rm d}x{\rm d}t\\[0.3 cm]
+ \dsp\sum_{\a=1}^N\:\:\iti\iR\iy\tilde b_\a\cdot\Big(\nx v + \ny v_{1,\a}\Big)
\phi_{1,\a}{\rm d}y{\rm d}x{\rm d}t\\[0.3 cm]
-\dsp\sum_{\a=1}^N\:\:\iti\iR\iy \div_y\Big(\tilde D_\a\Big(\nx v + \ny
v_{1,\a}\Big)\Big)\phi_{1,\a}{\rm d}y{\rm d}x{\rm d}t\\[0.3 cm]
+ \dsp\sum_{\a=1}^N\:\:\sum_{\b=1}^N\:\:\iti\iR\iy
\Pi_{\a\b}\v^*_\a\v_\b\Big(v_{1,\b}-v_{1,\a}\Big)\phi_{1,\a}(t,x,y){\rm d}y{\rm
d}x{\rm d}t = 0.
\end{array}
\end{equation}
The above expression is the variational formulation for the following PDE:
\begin{equation}
\left\{
\begin{array}{cc}
\dsp\tilde b_\a\cdot \Big(\ny v_{1,\a} + \nx v\Big) - \div_y \Big(\tilde D_\a
\Big(\nabla_y v_{1,\a}+\nx v\Big)\Big)\\[0.3 cm]
+ \dsp\sum_{\b=1}^N \Pi_{\a\b}\v^*_\a\v_\b(v_{1,\b}-v_{1,\a})=
\v_\a\v^*_\a\rho_\a b^*\cdot \nx v & \mbox{ in }Y, \\[0.2 cm]
y \to v_{1,\a}(y) & Y\mbox{-periodic,}
\end{array} \right.
\label{eq:vv1}
\end{equation}
for every $1\le\a\le N$. By the Fredholm result of Lemma \ref{lem:fh}, we have
the existence and uniqueness of $(v_{1,\a})_{1\le\a\le N}\in L^2((0,T)\times\R^d;\mathscr{H}(Y))$
if and only if the compatibility condition (\ref{eq:fhc}) is satisfied. Writing
down the compatibility condition for (\ref{eq:vv1}) yields the expression
(\ref{eq:drift}) for the drift velocity $b^*$. Also by linearity of
(\ref{eq:vv1}), we deduce that we can separate the slow and fast variables in
$v_{1,\a}$ as in (\ref{eq:2scl:corr}) with $(\omega_{i,\a})_{1\le\a\le N}$
satisfying the coupled cell problem (\ref{eq:cpb}).

In a second step we choose $\phi^\ve_{1,\a}\equiv0$ in (\ref{eq:vf:1}) and 
substitute (\ref{eq:in1}) for the initial data $v^\ve_\a(0,x)$, which yields 
\begin{equation}
\label{eq:vf:2}
\begin{array}{ll}
-\dsp\sum_{\a=1}^N\:\:\iti\iR\v^\ve_\a\v^{*\ve}_\a\rho_\a^\ve v^\ve_\a\frac{\partial
\phi^\ve}{\partial t}{\rm d}x{\rm d}t
-\dsp\sum_{\a=1}^N\:\:\iR\v^{*\ve}_\a\rho_\a^\ve u^{in}_\a(x)\phi^\ve(0,x){\rm d}x\\[0.2 cm]
+\dsp\frac{1}{\ve}\dsp\sum_{\a=1}^N\:\:\iti\iR\v^\ve_\a\v^{*\ve}_\a v^\ve_\a\rho_\a^\ve
b^*\cdot \nx \phi^\ve{\rm d}x{\rm d}t
-\frac{1}{\ve}\dsp\sum_{\a=1}^N\:\:\iti\iR v^\ve_\a\tilde b^\ve_\a
\cdot\nx\phi^\ve{\rm d}x{\rm d}t\\[0.2 cm]
-\dsp\frac{1}{\ve}\dsp\sum_{\a=1}^N\:\:\iti\iR \div\Big(\tilde b^\ve_\a\Big)
v^\ve_\a\phi^\ve{\rm d}x{\rm d}t
+\dsp\sum_{\a=1}^N\:\:\iti\iR\tilde D^\ve_\a \nabla
v^\ve_\a\cdot\nx\phi^\ve{\rm d}x{\rm d}t\\[0.2 cm]
+\dsp\frac{1}{\ve^2}\sum_{\a=1}^N\:\:\sum_{\b=1}^N\:\:\iti\iR
\Pi^\ve_{\a\b}\v^{*\ve}_\a\v^\ve_\b(v^\ve_\b-v^\ve_\a)\phi^\ve{\rm d}x{\rm d}t = 0.
\end{array}
\end{equation}
Using the expression (\ref{eq:divb}) for the divergence of $\tilde b_\a$ allows us 
to obtain
$$
\dsp-\frac{1}{\ve}\dsp\sum_{\a=1}^N\:\:\iti\iR \div\Big(\tilde b^\ve_\a\Big)
v^\ve_\a\phi^\ve{\rm d}x{\rm d}t
+\dsp\frac{1}{\ve^2}\sum_{\a=1}^N\:\:\sum_{\b=1}^N\:\:\iti\iR
\Pi^\ve_{\a\b}\v^{*\ve}_\a\v^\ve_\b(v^\ve_\b-v^\ve_\a)\phi^\ve{\rm d}x{\rm d}t
$$
\begin{equation}
\label{eq:vf:div}
\begin{array}{ll}
= \dsp\frac{1}{\ve^2}\sum_{\a=1}^N\:\:\sum_{\b=1}^N\:\:\iti\iR
\Big(\Pi^\ve_{\a\b}\v^{*\ve}_\a\v^\ve_\b v^\ve_\a - \Pi_{\a\b}^{\ve *}\v^\ve_\a\v^{*\ve}_\b
v^\ve_\a\Big)\phi^\ve{\rm d}x{\rm d}t\\[0.2 cm]
+\dsp\frac{1}{\ve^2}\sum_{\a=1}^N\:\:\sum_{\b=1}^N\:\:\iti\iR
\Big(\Pi^\ve_{\a\b}\v^{*\ve}_\a\v^\ve_\b v^\ve_\b - \Pi^\ve_{\a\b}\v^{*\ve}_\a\v^\ve_\b v^\ve_\a \Big)\phi^\ve{\rm d}x{\rm d}t=0.
\end{array}
\end{equation}
Thanks to (\ref{eq:vf:div}) all terms of order $\mathcal O (\ve^{-2})$ in
(\ref{eq:vf:2}) cancel each other. There are, however, terms of $\mathcal O
(\ve^{-1})$ in (\ref{eq:vf:2}) which still prevent us to pass to the limit
as $\ve\to0$. In order to remedy the situation, we introduce the following
auxiliary problem posed in the unit cell:
\begin{equation}
\label{eq:aux}
\left\{
\begin{array}{ll}
\dsp-\Delta \Xi = \dsp\sum_{\a=1}^N\:\Big(\v_\a\v^*_\a\rho_\a b^* - \tilde
b_\a\Big)  & \mbox{ in }Y, \\[0.2 cm]
y \to \Xi(y) & Y\mbox{-periodic.}
\end{array} \right.
\end{equation}
The above auxiliary problem is well-posed, thanks to our choice (\ref{eq:drift}) 
of the drift velocity and the chosen normalization (\ref{eq:norm}). We scale 
(\ref{eq:aux}) to the entire domain via the change of variables $y\to\ve^{-1}x$. 
The vector-valued function $\Xi^\ve(x) =\Xi(x/\ve)$ satisfies
\begin{equation}
\label{eq:aux:scl}
\left\{
\begin{array}{ll}
\dsp-\ve^2 \Delta \Xi^\ve =
\dsp\sum_{\a=1}^N\:\Big(\v^\ve_\a\v^{*\ve}_\a\rho_\a^\ve b^* - \tilde b^\ve_\a\Big) &
\mbox{ in }\R^d,\\[0.2 cm]
x \to \Xi^\ve & \ve Y\mbox{-periodic.}
\end{array} \right.
\end{equation}
Getting back to the variational formulation (\ref{eq:vf:2}), let us regroup the
problematic terms of order $\mathcal O (\ve^{-1})$:
$$
\dsp\frac{1}{\ve}\dsp\sum_{\a=1}^N\:\:\iti\iR v^\ve_\a\Big(\v^\ve_\a\v^{*\ve}_\a
v^\ve_\a\rho_\a^\ve b^* - b^\ve_\a\Big)\cdot \nx \phi^\ve{\rm d}x{\rm d}t
$$
\begin{equation}
\label{eq:vf:rerng}
\begin{array}{ll}
\dsp=-\frac{\ve^2}{\ve}\sum_{i=1}^d\iti\iR \Delta \Xi_i^\ve \frac{\partial
\phi^\ve}{\partial x_i}v^\ve_\a{\rm d}x{\rm d}t
+\frac{1}{\ve}\dsp\sum_{\b=1}^N\:\iti\iR(v^\ve_\a-v^\ve_\b)\tilde b^\ve_\b
\cdot\nx\phi^\ve{\rm d}x{\rm d}t\\[0.3 cm]
+\dsp\frac{1}{\ve}\dsp\sum_{\b=1}^N\:\iti\iR\v^\ve_\b\v^{*\ve}_\b\rho_\b^\ve(v^\ve_\b-v^\ve_\a)
b^*\cdot \nx \phi^\ve{\rm d}x{\rm d}t
\end{array}
\end{equation}
\begin{equation}
\label{eq:vf:sing}
\begin{array}{ll}
\dsp= \ve\sum_{i=1}^d\iti\iR\nabla \Xi_i^\ve\cdot\nabla\Big(\frac{\partial
\phi^\ve}{\partial x_i}\Big)v^\ve_\a{\rm d}x{\rm d}t +
\ve\sum_{i=1}^d\iti\iR\nabla \Xi_i^\ve\cdot\nabla v^\ve_\a \Big(\frac{\partial
\phi^\ve}{\partial x_i}\Big){\rm d}x{\rm d}t\\[0.3 cm]
\dsp+\frac{1}{\ve}\dsp\sum_{\b=1}^N\:\iti\iR\v^\ve_\b\v^{*\ve}_\b\rho_\b^\ve(v^\ve_\b-v^\ve_\a)
b^*\cdot \nx \phi^\ve{\rm d}x{\rm d}t\\[0.3 cm]
\dsp+\frac{1}{\ve}\dsp\sum_{\b=1}^N\:\iti\iR(v^\ve_\a-v^\ve_\b)\tilde b^\ve_\b
\cdot\nx\phi^\ve{\rm d}x{\rm d}t ,
\end{array}
\end{equation}
where we have used the scaled auxiliary problem (\ref{eq:aux:scl}). 
We can now pass to the limit in (\ref{eq:vf:sing}) since the sequences 
$(v^\ve_\b-v^\ve_\a)/\ve$ are bounded. Taking into consideration
(\ref{eq:vf:div}) and (\ref{eq:vf:sing}), the variational formulation
(\ref{eq:vf:2}) rewrites as
\begin{equation}
\label{eq:vf:3}
\begin{array}{ll}
-\dsp\sum_{\a=1}^N\:\:\iti\iR\v^\ve_\a\v^{*\ve}_\a\rho_\a^\ve v^\ve_\a\frac{\partial
\phi^\ve}{\partial t}{\rm d}x{\rm d}t
-\dsp\sum_{\a=1}^N\:\:\iR\v^{*\ve}_\a\rho_\a^\ve u^{in}_\a(x)\phi^\ve(0,x){\rm d}x\\[0.2 cm]
\dsp+\ve\sum_{i=1}^d\iti\iR\nabla \Xi_i^\ve\cdot\nabla\Big(\frac{\partial
\phi^\ve}{\partial x_i}\Big)v^\ve_\a{\rm d}x{\rm d}t +
\ve\sum_{i=1}^d\iti\iR\nabla \Xi_i^\ve\cdot\nabla v^\ve_\a \Big(\frac{\partial
\phi^\ve}{\partial x_i}\Big){\rm d}x{\rm d}t\\[0.2 cm]
\dsp+\frac{1}{\ve}\dsp\sum_{\b=1}^N\:\iti\iR\v^\ve_\b\v^{*\ve}_\b\rho_\b^\ve(v^\ve_\b-v^\ve_\a)
b^*\cdot \nx \phi^\ve{\rm d}x{\rm d}t
+\dsp\sum_{\a=1}^N\:\:\iti\iR\tilde D^\ve_\a \nabla
v^\ve_\a\cdot\nx\phi^\ve{\rm d}x{\rm d}t\\[0.2 cm]
\dsp+\frac{1}{\ve}\dsp\sum_{\b=1}^N\:\iti\iR(v^\ve_\a-v^\ve_\b)\tilde b^\ve_\b
\cdot\nx\phi^\ve{\rm d}x{\rm d}t=0.
\end{array}
\end{equation}
Using the compactness results from Theorem \ref{thm:3:2scl}, we
pass to the limit as $\ve\to0$ in the above variational formulation leading to:
\begin{equation}
\label{eq:vf:4}
\begin{array}{ll}
-\dsp\iti\iR v\frac{\partial \phi}{\partial t}{\rm d}x{\rm d}t 
-\dsp\sum_{\a=1}^N\:\iR\iy u^{in}_\a(x)\phi(0,x) \v^{*}_\a\rho_\a{\rm d}y{\rm d}x\\[0.2 cm]
+\dsp\sum_{\a=1}^N\:\iti\iR\iy\tilde D_\a \Big(\nabla v + \ny
v_{1,\a}\Big)\cdot\nx\phi{\rm d}y{\rm d}x{\rm d}t\\[0.2 cm]
\dsp+\sum_{i=1}^d\iti\iR\iy\ny\Xi_i\cdot\ny v_{1,\a} \frac{\partial
\phi}{\partial x_i}{\rm d}y{\rm d}x{\rm d}t\\[0.2 cm]
+\dsp\sum_{\b=1}^N\:\iti\iR\iy\v_\b\v^*_\b\rho_\b\Big(v_{1,\b}-v_{1,\a}\Big)
b^*\cdot \nx \phi{\rm d}y{\rm d}x{\rm d}t\\[0.2 cm]
+\dsp\sum_{\b=1}^N\:\iti\iR\iy\Big(v_{1,\a}-v_{1,\b}\Big)\tilde b_\b
\cdot\nx\phi{\rm d}y{\rm d}x{\rm d}t = 0.
\end{array}
\end{equation}
Substituting (\ref{eq:2scl:corr}) for $v_{1,\a}$ in
(\ref{eq:vf:4}), we obtain
\begin{equation}
\label{eq:vf:lim}
\begin{array}{ll}
-\dsp\iti\iR v\frac{\partial \phi}{\partial t}{\rm d}x{\rm d}t
-\dsp\sum_{\a=1}^N\:\iR u^{in}_\a(x)\phi(0,x)\iy \v^{*}_\a(y)\rho_\a(y){\rm d}y{\rm d}x\\[0.2 cm]
+\dsp\sum_{\a=1}^N\:\sum_{i,j=1}^d\:\iti\iR\frac{\partial v}{\partial
x_j}\frac{\partial \phi}{\partial x_i}\iy \tilde D_\a\Big(\nabla y_j  +
\ny\omega_{j,\a}\Big)\cdot\nabla y_i{\rm d}y{\rm d}x{\rm d}t\\[0.2 cm]
-\dsp\sum_{i,j=1}^d\iti\iR\frac{\partial v}{\partial x_j}\frac{\partial
\phi}{\partial x_i}\iy\Big(\Delta_y\Xi_i\Big) \omega_{j,\a}{\rm d}y{\rm d}x{\rm
d}t\\[0.2 cm]
+\dsp\sum_{\b=1}^N\:\sum_{i,j=1}^d\:\iti\iR\frac{\partial v}{\partial
x_j}\frac{\partial \phi}{\partial x_i}\iy\v_\b\v^*_\b\rho_\b \Big(\omega_{j,\b}
- \omega_{j,\a}\Big) b^*\cdot\nabla y_i{\rm d}y{\rm d}x{\rm d}t\\[0.2 cm]
+\dsp\sum_{\b=1}^N\:\sum_{i,j=1}^d\:\iti\iR\frac{\partial v}{\partial
x_j}\frac{\partial \phi}{\partial x_i}\iy\Big(\omega_{j,\a} -
\omega_{j,\b}\Big) \tilde b_\b\cdot\nabla y_i{\rm d}y{\rm d}x{\rm d}t = 0.
\end{array}
\end{equation}
Using the information from the auxiliary cell problem (\ref{eq:aux}) in
(\ref{eq:vf:lim}) and making a rearrangement similar to that of
(\ref{eq:vf:rerng}), we deduce that (\ref{eq:vf:lim}) is nothing but the
variational formulation for a scalar diffusion equation (\ref{eq:hom}) for
$v(t,x)$ with the entries of the diffusion matrix given by
$$
{\mathcal D}_{ij} = \dsp\sum_{\a=1}^N\:\iy\tilde D_\a\Big(\nabla y_j +
\ny\omega_{j,\a}\Big)\cdot\nabla y_i{\rm d}y
+\dsp\sum_{\a=1}^N\:\:\iy \omega_{j,\a}\Big(\v_\a\v^*_\a v^\ve_\a\rho_\a^\ve
b^* - b^\ve_\a\Big)\cdot e_i{\rm d}y.
$$
By integration by parts, it is clear that the diffusion matrix $\mathcal D$ 
is contracted with the Hessian matrix $\nabla\nabla v$, which is symmetric. 
Thus the non-symmetric part of $\mathcal D$ does not contribute to the homogenized 
equation (\ref{eq:hom}). So, the above expression for the diffusion matrix is
symmetrized:
\begin{equation}
\label{eq:disp:sym}
\begin{array}{ll}
{\mathcal D}_{ij} = \dsp\sum_{\a=1}^N\:\:\iy\tilde D_\a e_j\cdot e_i{\rm d}y
+\frac12 \Big\{\dsp\sum_{\a=1}^N\:\:\iy\Big(\tilde D_\a\ny\omega_{i,\a}\cdot
e_j + \tilde D_\a\ny\omega_{j,\a}\cdot e_i\Big){\rm d}y\Big\}\\[0.2 cm]
+\dsp \frac12 \Big\{\dsp\sum_{\a=1}^N\:\:\iy\Big(
\omega_{i,\a}(\v_\a\v^*_\a\rho_\a b^* -\tilde b_\a)\cdot e_j +
\omega_{j,\a}(\v_\a\v^*_\a\rho_\a b^* -\tilde b_\a)\cdot e_i\Big){\rm
d}y\Big\}.
\end{array}
\end{equation}
To obtain the desired expression (\ref{eq:disp}) for the diffusion matrix,
we consider the variational formulation for the cell problem
(\ref{eq:cpb}) with test functions $(\psi_\a)_{1\le\a\le N}$ 
\begin{equation}
\label{eq:vf:cpb1}
\begin{array}{ll}
\dsp\sum_{\a=1}^N\:\:\iy\Big(\tilde b_\a\cdot\ny\omega_{i,\a}\Big)\psi_\a{\rm
d}y+\dsp\sum_{\a=1}^N\:\:\sum_{\b=1}^N\:\:\iy\Pi_{\a\b}
\v^*_\a\v_\b\Big(\omega_{i,\b}-\omega_{i,\a}\Big)\psi_\a{\rm d}y\\[0.2 cm]
+\dsp\sum_{\a=1}^N\:\:\iy\tilde D_\a
\Big(\nabla_y\omega_{i,\a}+e_i\Big)\cdot\ny\psi_\a{\rm
d}y=\dsp\sum_{\a=1}^N\:\:\iy\Big(\v_\a\v^*_\a\rho_\a b^* - \tilde
b_\a\Big)\cdot e_i\psi_\a{\rm d}y.
\end{array}
\end{equation}
In (\ref{eq:vf:cpb1}) we first choose the test function $(\psi_\a)=(\omega_{j,\a})$. 
Similarly, in (\ref{eq:vf:cpb1}) for $j$ instead of $i$, we choose the test function 
$(\psi_\a)=(\omega_{i,\a})$. This leads to
\begin{equation}
\label{eq:vf:cpb2}
\begin{array}{ll}
\dsp\frac12 \Big\{\dsp\sum_{\a=1}^N\:\:\iy\Big(
\omega_{i,\a}(\v_\a\v^*_\a\rho_\a b^* -\tilde b_\a)\cdot e_j +
\omega_{j,\a}(\v_\a\v^*_\a\rho_\a b^* -\tilde b_\a)\cdot e_i\Big){\rm
d}y\Big\}\\[0.2 cm]
=\dsp\sum_{\a=1}^N\:\:\iy\tilde D_\a\ny\omega_{i,\a}\cdot\ny\omega_{j,\a}{\rm
d}y+\frac12 \Big\{\dsp\sum_{\a=1}^N\:\:\iy\Big(\tilde D_\a\ny\omega_{i,\a}\cdot
e_j + \tilde D_\a\ny\omega_{j,\a}\cdot e_i\Big){\rm d}y\Big\}\\[0.2 cm]
\dsp-\frac12\Big\{\dsp\sum_{\a=1}^N\:\:\iy\omega_{i,\a}\omega_{j,\a}\div_y\tilde
b_\a{\rm d}y\Big\}\\[0.2 cm]
\dsp+\frac12 \Big\{\dsp\sum_{\a=1}^N\:\:\sum_{\b=1}^N\:\:\iy\Big(\Pi_{\a\b}
\v^*_\a\v_\b(\omega_{i,\b}-\omega_{i,\a})\omega_{j,\a} + \Pi_{\a\b}
\v^*_\a\v_\b(\omega_{j,\b}-\omega_{j,\a})\omega_{i,\a}\Big){\rm d}y\Big\}.
\end{array}
\end{equation}
Using formula (\ref{eq:divb}) for the divergence of $\tilde b_\a$ in 
(\ref{eq:vf:cpb2}), its right hand side simplifies as
\begin{equation}
\label{eq:vf:cpb3}
\begin{array}{ll}
\dsp\sum_{\a=1}^N\:\:\iy\tilde D_\a\ny\omega_{i,\a}\cdot\ny\omega_{j,\a}{\rm
d}y+\frac12 \Big\{\dsp\sum_{\a=1}^N\:\:\iy\Big(\tilde D_\a\ny\omega_{i,\a}\cdot
e_j + \tilde D_\a\ny\omega_{j,\a}\cdot e_i\Big){\rm d}y\Big\}\\[0.2 cm]
\dsp-\frac12\dsp\sum_{\a=1}^N\:\:\sum_{\b=1}^N\:\:\iy
\v^*_\a\v_\b\Pi_{\a\b}\Big(\omega_{i,\a} - \omega_{i,\b}\Big)\Big(\omega_{j,\a}
- \omega_{j,\b}\Big){\rm d}y.
\end{array}
\end{equation}
Plugging (\ref{eq:vf:cpb3}) in the symmetrized formula (\ref{eq:disp:sym}) 
leads to the desired equation (\ref{eq:disp}). 
Eventually the scalar homogenized equation (\ref{eq:hom}) has a unique 
solution since, by virtue of Lemma \ref{lem:disp}, the dispersion matrix 
is positive definite. This guarantees that the entire sequence $v^\ve_\a$ 
converges to $v$, for $1\le\a\le N$, and not merely a subsequence as in 
Theorem \ref{thm:3:2scl}.
\end{proof}

\section{Adsorption in Porous Media}
\label{sec:apm}

In this section, we give a generalization of our previous result in a
more applied context.
Our goal is to upscale a model of multicomponent transport in an highly
heterogeneous porous medium in presence of adsorption reaction at the
fluid-pore interface. In \cite{AR:07}, the authors study the homogenization of
one single scalar convection-diffusion-reaction equation posed in an
$\ve$-periodic infinite porous medium:
\begin{equation}
\label{eq:ar07}
\left\{
\begin{array}{ll}
\dsp\rho^\ve\frac{\partial u^\ve}{\partial t} + \frac{1}{\ve}b^\ve\cdot\nabla
u^\ve - \div(D^\ve\nabla u^\ve) + \frac{1}{\ve^2} c^\ve u^\ve = 0 & \mbox{ in
}(0,T)\times\Omega_\ve,\\[0.2 cm]
\dsp-D^\ve\nabla u^\ve\cdot n = \frac{1}{\ve}\kappa u^\ve & \mbox{ on
}(0,T)\times\partial\Omega_\ve.
\end{array}
\right.
\end{equation}

Typically, an $\ve$-periodic infinite porous medium is built out of $\R^d$
($d=2$ or $3$, being the space dimension) by removing a periodic distribution
of solid obstacles which, after rescaling, are all similar to the unit obstacle
$\Sigma^0$. More precisely, let $Y = ]0,1[^d$ be the unit periodicity cell. Let
us consider a smooth partition $Y = \Sigma^0 \cup Y^0$ where $\Sigma^0$ is the
solid part and $Y^0$ is the fluid part. The fluid part (extended by
periodicity) is assumed to  be a smooth connected open subset whereas no
particular assumptions are made on the solid part. For each multi-index
$j\in\mathbb{Z}^d$, we define $Y^j_\ve := \ve(Y^0+j)$, $\Sigma^j_\ve :=
\ve(\Sigma^0+j)$, $S^j_\ve := \ve(\partial\Sigma^0+j)$, the periodic porous
medium $\Omega_\ve := \displaystyle
\cup_{j\in\mathbb{Z}^d} Y^j_\ve$ and the $(d-1)-$dimensional surface
$\partial\Omega_\ve := \cup_{j\in\mathbb{Z}^d} S^j_\ve$.

In this section, we generalize the results of \cite{AR:07} to the
multicomponent case. We consider the following weakly coupled cooperative
parabolic system with Neumann boundary condition at the fluid-pore interface.
\begin{equation}
\label{eq:cdb}
\left\{
\begin{array}{cc}
\dsp\rho^\ve_\a \frac{\partial u^\ve_\a}{\partial t} + \frac1\ve
b^\ve_\a\cdot\nabla u^\ve_\a - \div(D^\ve_\a\nabla u^\ve_\a) = 0&\mbox{ in
}(0,T)\times\Omega_\ve,\\[0.2 cm]
\dsp- D^\ve_\a\nabla u^\ve_\a\cdot n = \frac1{\ve}\sum_{\b=1}^N
\Pi^\ve_{\a\b}u^\ve_\b &\mbox{ on }(0,T)\times\partial\Omega_\ve,\\[0.2 cm]
\dsp u^\ve_\a(0,x) = u^{in}_\a(x)& \mbox{ in }\Omega_\ve .
\end{array}
\right.
\end{equation}

\begin{remark}
Note the different scaling in front of the surface reaction terms. 
It is of order $\ve^{-1}$ because it balances a flux rather than a 
diffusive term, as in the previous model of Section \ref{sec:model}. 
As usual, by the change of variable $(\tau,y)\to(\ve^{-2}t,\ve^{-1}x)$ 
all singular powers of $\ve$ disappears in (\ref{eq:cdb}) written in 
the $(\tau,y)$ variables.
\end{remark}

The hypotheses on the coefficients in (\ref{eq:cdb}) are exactly the
same as in Section \ref{sec:model}. As before it is impossible to
obtain uniform (in $\ve$) estimates on the solutions $u^\ve_\a$ of
(\ref{eq:cdb}). As was done in Section \ref{sec:fp}, we employ the method of
factorization by introducing a new unknown:
$$
v^\ve_\a(t,x)=\exp{(\lambda
t/\ve^2)}\frac{u^\ve_\a(t,x)}{\v_\a\Big(\frac{x}{\ve}\Big)},
$$
where $(\lambda,\v_\a)$ and $(\lambda,\v^*_\a)$ are the principal eigenpairs
associated with the (new) following spectral problems respectively:
\begin{equation}
\label{eq:scp:b}
\left\{
\begin{array}{ll}
\dsp b_\a(y) \cdot \ny \v_\a - \div_y \Big(D_\a \ny \v_\a\Big) = \lambda\rho_\a
\v_\a & \mbox{ in } Y^0, \\[0.2 cm]
\dsp- D_\a\ny\v_\a\cdot n = \sum_{\b=1}^N \Pi_{\a\b} \v_\b & \mbox{ on
}\partial\Sigma^0,\\[0.2 cm]
y \to \v_\a(y) & Y\mbox{-periodic.}
\end{array} \right.
\end{equation}
\begin{equation}
\label{eq:ascp:b}
\left\{
\begin{array}{ll}
\dsp -\div_y(b_\a \v^*_\a) - \div_y \Big(D_\a \ny \v^*_\a\Big) = \lambda\rho_\a
\v^*_\a & \mbox{ in } Y^0, \\[0.2 cm]
\dsp- D_\a\ny\v^*_\a\cdot n - b_\a(y)\cdot n \v^*_\a = \sum_{\b=1}^N
\Pi^*_{\a\b} \v^*_\b & \mbox{ on }\partial\Sigma^0,\\[0.2 cm]
y \to \v^*_\a(y) & Y\mbox{-periodic.}
\end{array} \right.
\end{equation}
Proposition \ref{prop:spec}, which guarantees the existence of
principal eigenpairs for the spectral problems (\ref{eq:scp})-(\ref{eq:ascp}), 
carries over to the above spectral problems (\ref{eq:scp:b})-(\ref{eq:ascp:b}) as
well. This is apparent from the proofs in \cite{Sw:92, MS:95}. The normalization (ensuring uniqueness of the eigenfunctions) that we choose is:
$$
\dsp \sum_{\a=1}^N\: \int_{Y^0} \v_\a\v^*_\a \rho_\a \,{\rm d}y = 1.
$$
As in Section \ref{sec:fp} it is a matter of simple algebra to obtain 
the factorized system for (\ref{eq:cdb}) with the new unknown which is, 
for each $1 \le \alpha \le N$,
\begin{equation}
\label{eq:cdb1}
\left\{
\begin{array}{ll}
\dsp\v^\ve_\a\v^{*\ve}_\a\rho^\ve_\a\frac{\partial v^\ve_\a}{\partial t} +
\frac{1}{\ve}\tilde b^\ve_\a \cdot \nabla v^\ve_\a - \div\left(\tilde D^\ve_\a
\nabla v^\ve_\a\right) = 0 & \mbox{ in }(0,T)\times\Omega_\ve,\\[0.2 cm]
\dsp-\tilde D^\ve_\a\nabla v^\ve_\a\cdot n =
\frac{1}{\ve}\dsp\sum_{\b=1}^N\:\:\Pi^\ve_{\a\b}\v^{*\ve}_\a\v^\ve_\b(v^\ve_\b-v^\ve_\a)
& \mbox{ on }(0,T)\times\partial\Omega_\ve,\\[0.2 cm]
\dsp v^\ve_\a(0,x) = \frac{u^{in}_\a(x)}{\v_\a\Big(\frac{x}{\ve}\Big)}&
\mbox{ in } \Omega_\ve,
\end{array}
\right.
\end{equation}
where the convective fields $\tilde b_\a$ and diffusion matrices $\tilde D_\a$ 
are given by the same formulae (\ref{eq:b1}) and (\ref{eq:D1}). A proof, 
completely similar to that of Lemma \ref{lem:apriori}, yields 
the following a priori estimates 
(\ref{eq:cdb1}):
\begin{equation}
\label{eq:cdb1:ap}
\begin{array}{ll}
\dsp\sum_{\a=1}^N\:\:\Big\|v^\ve_\a\Big\|_{L^\infty((0,T);L^2(\Omega_\ve))}+\sum_{\a=1}^N\:\:\Big\|\nabla
v^\ve_\a\Big\|_{L^2((0,T)\times\Omega_\ve)}\\[0.4 cm]
+\dsp\sum_{\a=1}^N\:\:\sum_{\b=1}^N\:\:\sqrt{\ve}\Big\|\frac1\ve(v^\ve_\a-v^\ve_\b)\Big\|_{L^2((0,T)\times\partial\Omega_\ve)}\leq
C\sum_{\a=1}^N\:\:\|v^{in}_\a\|_{L^2(\R^d)}.
\end{array}
\end{equation}

\begin{remark}
\label{rem:surf}
Since the $(d-1)$ dimensional measure of the periodic surface
$\partial\Omega_\ve$ is of order $\mathcal O(\ve^{-1})$, a bound of the type
$\sqrt{\ve}\|z_\ve\|_{L^2(\partial\Omega_\ve)}\le C$ means that the sequence 
$z_\ve$ is bounded on the surface $\partial\Omega_\ve$.
\end{remark}

In the a priori estimates (\ref{eq:cdb1:ap}), we have bounds in function spaces
defined on the periodic surface $\partial\Omega_\ve$. In order to speak of the
convergence of sequences in such function spaces, we need to generalize the
Definition \ref{def:2scl} of two-scale convergence with drift for periodic
surfaces. This generalization was introduced in \cite{Hu:13}. We state this 
definition together with the corresponding compactness result (the proof of 
which is similar to that of Theorem 9.1 in \cite{Al:08}). 

\begin{lemma}
\label{lem:2ds}
Let $b^*\in\R^d$ be a constant vector. Suppose that $u_\ve(t,x)$ is a sequence 
of functions uniformly bounded in $L^2((0,T)\times\partial\Omega_\ve)$ in the
sense that
$$
\sqrt{\ve}\|u_\ve\|_{L^2((0,T)\times\partial\Omega_\ve)}\le C .
$$
Then, there exists a subsequence, still denoted by $u_\ve(t,x)$, and a function $u_0(t,x,y)\in L^2((0,T)\times\R^d\times \partial\Sigma^0)$ such that 
\begin{equation}
\label{eq:defn2ds}
\begin{array}{cc}
\dsp\lim_{\ve\to0}\ve\iti\int_{\partial\Omega_\ve}
u_\ve(t,x)\phi\Big(t,x-\frac{b^*}{\ve}t,\frac{x}{\ve}\Big){\rm d}\sigma_\ve(x){\rm d}t \\[0.2 cm]
\dsp= \iti\iR\int_{\partial\Sigma^0} u_0(t,x,y)\phi(t,x,y){\rm d}\sigma(y){\rm d}x{\rm d}t ,
\end{array}
\end{equation}
for any function $\phi(t,x,y)\in C^\infty_c((0,T)\times\R^d;C^\infty_\#(Y))$.
\end{lemma}

In (\ref{eq:defn2ds}), ${\rm d}\sigma_\ve(x)$ and ${\rm d}\sigma(y)$ denote the standard surface measures on $\partial\Omega_\ve$ and $\partial\Sigma^0$ respectively. We denote this convergence on periodic surfaces in moving coordinates by $u_\ve \tss u_0$.

\begin{remark}
\label{rem:2ds:1}
Let $u_\ve(t,x)$ be a sequence of functions defined on $(0,T)\times\Omega_\ve$. 
Let $\gamma$ be the trace operator, i.e., $\gamma u = u|_{\partial\Omega_\ve}$. 
Suppose that we have a well-defined sequence of associated trace functions 
$\gamma u_\ve(t,x)$ on $(0,T)\times\partial\Omega_\ve$. If $u_\ve\ts u_0$ and 
$\gamma u_\ve\tss v_0$ with the same drift velocity for both convergences, 
then $\gamma u_0 = v_0$ i.e., $\gamma u_0 = u_0|_{\partial\Sigma^0}=v_0$ 
(see \cite{ADH} for details). 
In the sequel we systematically identify the ``bulk'' and ``surface'' two-scale 
limits. 
\end{remark}

We now define the homogenized velocity which is chosen as the constant drift 
in the definition of two-scale convergence with drift:
\begin{equation}
\label{eq:cdb:drift}
\dsp b^* = \sum_{\a=1}^N\,\int_{Y^0} \tilde b_\a(y){\rm d}y.
\end{equation}

\begin{theorem}
\label{thm:cdb:hom}
Let $(v^\ve_\a)_{1\le\a\le N}\in L^2((0,T);H^1(\Omega_\ve))^N$ be the 
sequence of solutions of (\ref{eq:cdb1}). Let $b^*\in\R^d$ be given by
(\ref{eq:cdb:drift}). There exist $v\in L^2((0,T);H^1(\R^d))$ and
$\omega_{i,\a}\in H_\#^1(Y^0)$, for $1\le\a\le N$ and $1\le i\le d$, 
such that $v^\ve_\a$ two-scale converges with drift $b^*$, as $\ve \to 0$, 
in the following sense:
\begin{equation}
\label{eq:3:brxn:2scl}
\left\{
\begin{array}{ll}
v^\ve_\a \ts v,\\[0.2 cm]
\dsp\nabla v^\ve_\a \ts \nx v + \ny\Big(\sum_{i=1}^d
\omega_{i,\a}\frac{\partial v}{\partial x_i}\Big),\\[0.3 cm]
\dsp\frac{1}{\ve}\Big(v^\ve_\a-v^\ve_\b\Big) \tss \sum_{i=1}^d
\Big(\omega_{i,\a} - \omega_{i,\b}\Big)\frac{\partial v}{\partial x_i},
\end{array}\right.
\end{equation}
for every $1\le\a,\b\le N$. The two-scale limit $v(t,x)$ in
(\ref{eq:3:brxn:2scl}) satisfies the following homogenized equation:
\begin{equation}
\label{eq:3:brxn:hom}
\left\{
\begin{array}{ll}
\dsp\frac{\partial v}{\partial t} - \div({\mathcal D}\nabla v) = 0 &
\textrm{in}\:\: (0,T)\times\R^d, \\[0.3cm]
v(0,x) = \dsp\sum_{\a=1}^N u^{in}_\a(x) \dsp\int_{Y^0}\rho_\a(y) \v^*_\a(y){\rm d}y &
\textrm{in}\:\: \R^d,
\end{array} \right.
\end{equation}
where the dispersion tensor $\mathcal D$ is given by
\begin{equation}
\label{eq:3:brxn:disp}
\begin{array}{ll}
{\mathcal D}_{ij}=\dsp\sum_{\a=1}^N\:\:\int_{Y^0}\tilde
D_\a\Big(\nabla_y\omega_{i,\a} + e_i\Big)\cdot\Big(\nabla_y\omega_{j,\a} +
e_j\Big){\rm d}y\\[0.5 cm]
-\dsp\frac12\sum_{\a,\b=1}^N\:\:\int_{\partial\Sigma^0}
\v^*_\a\v_\b\Pi_{\a\b}\Big(\omega_{i,\a} - \omega_{i,\b}\Big)\Big(\omega_{j,\a}
- \omega_{j,\b}\Big){\rm d}\sigma(y)
\end{array}
\end{equation}
and the components $(\omega_{i,\a})_{1\le\a\le N}$, for every $1\leq i\leq d$,
are the solutions of the cell problems
\begin{equation}
\label{eq:3:brxn:cellpb}
\left\{
\begin{array}{lll}
\tilde b_\a(y)\cdot \Big(\ny\omega_{i,\a} + e_i\Big) - \div_y \Big(\tilde D_\a
\Big(\nabla_y\omega_{i,\a}+e_i\Big)\Big)= \v_\a\v^*_\a\rho_\a b^*\cdot e_i&
\textrm{in}\:\: Y^0, \\[0.3cm]
-\tilde D_\a(\ny \omega_{i,\a}+e_i)\cdot n = \dsp\sum_{\b=1}^N\: \Pi_{\a\b}
\v^*_\a\v_\b\Big(\omega_{i,\b}-\omega_{i,\a}\Big) & \textrm{on} \: \:
\partial\Sigma^0, \\[0.3cm]
y \to \omega_{i,\a} & Y\textrm{-periodic.}
\end{array} \right.
\end{equation}
\end{theorem}

\begin{proof}
As we have $L^2$ bounds on the solution sequence, we have the existence of a
subsequence and a two-scale limit, say $(v_\a)_{1\le\a\le N}\in
L^2((0,T);L^2(\R^d))^N$ such that
\begin{equation}
\label{eq:cpb:2scl}
v^\ve_\a\ts v_\a
\end{equation}
for every $1\le\a\le N$. For $w\in H^1(\Omega_\ve)$, consider the following
Poincar\'e type inequality derived in \cite{Co:85}:
\begin{equation}
\label{eq:Co:poin}
\|w\|^2_{L^2(\Omega_\ve)}\le C\Big(\ve^2\|\nabla w\|^2_{L^2(\Omega_\ve)} +
\ve\|w\|^2_{L^2(\partial\Omega_\ve)}\Big).
\end{equation}
Taking $\dsp w = \frac{1}{\ve}\Big(v^\ve_\a-v^\ve_\b\Big)$, we deduce from
(\ref{eq:Co:poin}) and a priori estimates (\ref{eq:cdb1:ap}) that
\begin{equation}
\label{eq:cpb:diff}
\dsp\sum_{\a=1}^d\sum_{\b=1}^d\|v^\ve_\a-v^\ve_\b\|_{L^2((0,T)\times\Omega_\ve)}\le
C\,\ve.
\end{equation}
The above estimate (\ref{eq:cpb:diff}) implies that the limits obtained in
(\ref{eq:cpb:2scl}) do match i.e., $v_\a=v$ for every $1\le\a\le N$. The $H^1$
a priori estimate in space as in (\ref{eq:cdb1:ap}) does imply that $v\in
L^2((0,T);H^1(\R^d))$ and that there exist limits $v_{1,\a}\in
L^2((0,T)\times\R^d;H^1_\#(Y^0))$ such that
\begin{equation}
\label{eq:cdb:corr}
\nabla v^\ve_\a \ts \nx v + \ny v_{1,\a}
\end{equation}
for every $1\le\a\le N$. In order to arrive at the two-scale limit of the
coupled term on the boundary, we use Lemma \ref{lem:tech} below. 
Taking $\phi$ from (\ref{eq:tech:pb}) as the test function,
consider the following expression with the coupled term:
$$
\ve \int_0^T\int_{\partial\Omega_\ve} \frac{1}{\ve}\Big( v^\ve_\a -
v^\ve_\b\Big) \phi\left(t,x-\frac{b^* t}{\ve},\frac{x}{\ve}\right) {\rm
d}\sigma_\ve(x) {\rm d}t
$$
$$
=
\int_0^T\int_{\Omega_\ve} \div\left((v^\ve_\a - v^\ve_\b)
\theta\left(t,x-\frac{b^* t}{\ve},\frac{x}{\ve}\right)\right) {\rm d}x {\rm d}t
$$
$$
= \int_0^T\int_{\Omega_\ve}\Big(\nabla v^\ve_\a - \nabla v^\ve_\b\Big)
\cdot \theta\left(t,x-\frac{b^* t}{\ve},\frac{x}{\ve}\right){\rm d}x {\rm d}t
$$
$$
+ \int_0^T\int_{\Omega_\ve}\Big(v^\ve_\a - v^\ve_\b\Big) 
\left(\div_x \theta\right)\left(t,x-\frac{b^*t}{\ve},\frac{x}{\ve}\right) {\rm d}x {\rm d}t
$$
$$
\ts \int_0^T\int_{\mathbb{R}^d}\int_{Y^0} \Big(\nabla_y v_{1,\a} - \nabla_y
v_{1,\b}\Big) \cdot \theta{\rm d}y {\rm d}x {\rm d}t
$$
$$
= \int_0^T\iR\int_{\partial\Sigma^0} \Big( v_{1,\a} - v_{1,\b}\Big) \theta
\cdot n {\rm d}\sigma(y) {\rm d}x {\rm d}t =
\displaystyle\int_0^T\iR\int_{\partial\Sigma^0} \Big( v_{1,\a} - v_{1,\b}\Big)
\phi {\rm d}\sigma(y) {\rm d}x {\rm d}t,
$$
which implies that
$$
\dsp\frac{1}{\ve}\Big(v^\ve_\a-v^\ve_\b\Big) \tss v_{1,\a} - v_{1,\b}\:\:\mbox{
for every }1\le\a,\b\le N.
$$
The rest of the proof is completely similar to the proof of Theorem \ref{thm:hom}. 
We safely leave it to the reader. 
\end{proof}

We finish by stating a technical lemma which was useful in the proof of 
Theorem \ref{thm:cdb:hom}.

\begin{lemma}
\label{lem:tech}
For a function $\phi(t,x,y) \in
L^2((0,T)\times\mathbb{R}^d\times\partial\Sigma^0)$ such that
\begin{equation}
\label{eq:tech:comp}
\int\limits_{\partial\Sigma^0} \phi(t,x,y) {\rm d}\sigma(y) = 0 \hspace{1
cm}\forall\:(t,x)\in(0,T)\times \mathbb{R}^d,
\end{equation}
there exists a vector field
$\theta(t,x,y) \in L^2((0,T)\times\mathbb{R}^d;L^2_\#(Y^0))^d$ such that
\begin{equation}
\label{eq:tech:pb}
\left\{
 \begin{array}{ll}
  \div_y \theta = 0 & \textrm{in} \: \: Y^0,\\[0.2cm]
  \theta \cdot n = \phi & \textrm{on} \: \: \partial\Sigma^0,\\[0.2 cm]
  y \to \theta(t,x,y) & Y\mbox{-periodic.}
 \end{array}\right.
\end{equation}
\end{lemma}
\begin{proof}
Consider the following stationary diffusion problem posed in the unit cell:
$$
\begin{array}{ll}
\Delta_y\xi(y)=0 & \mbox{ in }Y^0,\\[0.2 cm]
\nabla_y\xi\cdot n = \phi& \mbox{ on }\partial\Sigma^0,
\end{array}
$$
with $Y$-periodic boundary conditions and the Neumann data $\phi$ satisfying
(\ref{eq:tech:comp}). The existence and uniqueness of $\xi\in H^1_{\#}(Y^0)/\R$
is guaranteed for the above problem as (\ref{eq:tech:comp}) is indeed the
compatibility condition from the Fredholm alternative. Choosing
$\theta=\nabla_y\xi$ gives one possible solution for (\ref{eq:tech:pb}).
\end{proof}

\medskip

{\bf Acknowledgements.} This work was initiated during the PhD thesis of H.H.
This work was partially supported by the GdR MOMAS from the CNRS. G.A. is a member
of the DEFI project at INRIA Saclay Ile-de-France. H.H. acknowledges the
support of the ERC grant MATKIT.


\signga

\signhh

\end{document}